\documentclass[a4paper, 12pt]{article}

\usepackage{amsmath} 
\usepackage{theorem}
\usepackage{amssymb}
\usepackage{amsfonts}
\usepackage{latexsym}
\usepackage{amscd}
\usepackage{diagramb}
\input GrCalc3.sty
\usepackage{graphicx}
\usepackage{enumerate}

\usepackage{stmaryrd}

\DeclareMathAlphabet{\mathpzc}{OT1}{pzc}{m}{it}

\newcommand{\mo}{{\mathcal M}}
\newcommand{\ca}{{\mathcal C}}
\newcommand{\otb}{{\overline{\otimes}}}
\newcommand{\Do}{{\mathcal D}}

\usepackage{xspace,colortbl}
\usepackage{color}

\definecolor{verde}{rgb}{0.,0.7,0.}
\definecolor{indigo}{rgb}{.18, .34, .78}
\definecolor{indigo1}{rgb}{.18, .24, .78}
\definecolor{indigo2}{rgb}{.18, .14, .78}
\definecolor{indigo3}{rgb}{.18, 0., .78}
\definecolor{rojo}{rgb}{1,0,0}
\definecolor{negro}{rgb}{0,0,0}
\definecolor{lila}{rgb}{.46, .16, .78}
\definecolor{lila1}{rgb}{.46, .16, .86}
\definecolor{lila2}{rgb}{.56, .16, .86}
	\definecolor{lila3}{rgb}{.63, .16, .78}
\definecolor{lila4}{rgb}{.7, .16, .78}
\definecolor{lila5}{rgb}{.78, .26, .78}
\definecolor{lila6}{rgb}{.6, 0., .78}

\theoremstyle{plain}
\theoremheaderfont{\normalfont\bfseries}

\newtheorem{thm}{Theorem}[section]

\newtheorem{lma}[thm]{Lemma}
\newtheorem{cor}[thm]{Corollary}
\newtheorem{defn}[thm]{Definition}

\theorembodyfont{\rmfamily}

\newtheorem{rem}[thm]{Remark}
\newtheorem{prop}[thm]{Proposition}

\newcommand{\qed}{\hfill\quad\fbox{\rule[0mm]{0,0cm}{0,0mm}}  \par\bigskip}

\def\ol{\overline}

\newcommand{\x}{\mbox{-}}
\newcommand{\R}{{\mathcal R}}

\newcommand{\Cc}{{\mathfrak C}}

\newcommand{\Del}{\boxtimes}
\newcommand{\comp}{\circ}
\newcommand{\iso}{\cong}
\newcommand{\ot}{\otimes}

\newcommand{\C}{{\mathcal C}}
\newcommand{\M}{{\mathcal M}}
\newcommand{\D}{{\mathcal D}}
\newcommand{\F}{{\mathcal F}}
\newcommand{\G}{{\mathcal G}}
\newcommand{\Oo}{{\mathcal O}}
\newcommand{\HH}{{\mathcal H}}
\newcommand{\N}{{\mathcal N}}
\newcommand{\A}{{\mathcal A}}

\newcommand{\E}{{\mathcal E}}

\def\ul{\underline}
\def\dul#1{\underline{\underline{#1}}}

\newcommand{\Ev}{\rm Ev}

\newcommand{\Ll}{{\mathcal L}}
\newcommand{\Pp}{{\mathcal P}}
\newcommand{\Qq}{{\mathcal Q}}

\newcommand{\crta}{\overline}
\newcommand{\ev}{\rm ev}

\newcommand{\Id}{\operatorname {Id}}

\newcommand{\Ker}{\operatorname {Ker}}
\newcommand{\Hom}{\operatorname {Hom}}

\newcommand{\Fun}{\operatorname {Fun}}

\def\Nn{{\mathbb N}}

\newcommand{\Pic}{\operatorname{Pic}}
\newcommand{\BrPic}{\operatorname{BrPic}}

\newcommand{\Mod}{\operatorname{Mod}}
\newcommand{\Bimod}{\operatorname{Bimod}}
\newcommand{\Inv}{\operatorname{Inv}}

\newcommand{\AUT}{\operatorname {AUT}}
\newcommand{\Nat}{\operatorname{Nat}}
\newcommand{\Pseud}{\operatorname {Pseud}}

\newcommand{\im}{{\rm Im}\,}

\newcommand{\cref}[1]{Cor.~\ref{c:#1}}

\newcommand{\lelabel}[1]{\label{le:#1}}
\newcommand{\leref}[1]{Lemma~\ref{le:#1}}
\newcommand{\eqlabel}[1]{\label{eq:#1}}
\newcommand{\equref}[1]{(\ref{eq:#1})}
\newcommand{\thlabel}[1]{\label{th:#1}}
\newcommand{\thref}[1]{Theorem~\ref{th:#1}}
\newcommand{\delabel}[1]{\label{de:#1}}
\newcommand{\deref}[1]{Definition~\ref{de:#1}}
\newcommand{\prlabel}[1]{\label{pr:#1}}
\newcommand{\prref}[1]{Proposition~\ref{pr:#1}}
\newcommand{\colabel}[1]{\label{co:#1}}
\newcommand{\coref}[1]{Corollary~\ref{co:#1}}

\newcommand{\rmlabel}[1]{\label{rm:#1}}
\newcommand{\rmref}[1]{Remark~\ref{rm:#1}}
\newcommand{\selabel}[1]{\label{se:#1}}
\newcommand{\seref}[1]{Section~\ref{se:#1}}
\newcommand{\sslabel}[1]{\label{ss:#1}}
\newcommand{\ssref}[1]{Subsection~\ref{ss:#1}}

\begin{document}

\title{Eilenberg-Watts Theorem for 2-categories and quasi-monoidal structures for module categories over bialgebroid categories}

\author{Bojana Femi\'c \vspace{6pt} \\
{\small Facultad de Ingenier\'ia, \vspace{-2pt}}\\
{\small  Universidad de la Rep\'ublica} \vspace{-2pt}\\
{\small  Julio Herrera y Reissig 565, \vspace{-2pt}}\\
{\small  11 300 Montevideo, Uruguay}}

\date{}
\maketitle

\begin{abstract}
We prove Eilenberg-Watts Theorem for 2-categories of the representation categories $\C\x\Mod$ of finite tensor categories $\C$. For a consequence we 
obtain that any autoequivalence of $\C\x\Mod$ is given by tensoring with 
a representative of some class in the Brauer-Picard group $\BrPic(\C)$. We introduce bialgebroid categories over $\C$ and a cohomology over a symmetric 
bialgebroid category. This cohomology turns out to be a generalization of the one we developed in a previous paper and moreover, an analogous 
Villamayor-Zelinsky sequence exists in this setting. In this context, for a symmetric bialgebroid category $\A$, we interpret the middle cohomology group 
appearing in the third level of the latter sequence. We obtain a group of quasi-monoidal structures on the representation category $\A\x\Mod$. 

\bigbreak

\end{abstract}

\section{Introduction}

In our previous paper \cite{Femic2} we introduced a cohomology over a symmetric finite tensor category $\C$ and we constructed an infinite exact sequence 
a la Villamayor-Zelinsky which contains three types of cohomology groups. These three types of groups repeat periodically in the sequence and we consider 
the sequence so that in each level there are three groups of different types. In \cite{Femic3} we interpreted the middle term in the second level of this sequence, 
obtaining the group of quasi $\C$-coring categories. These results generate to the setting of symmetric finite tensor categories our results from \cite{CF}, 
which were done in the context of commutative algebras over a commutative ring $R$. In the author's Ph.D. thesis \cite{th} and a paper which emerged out of it together with Stefaan Caenepeel 
we proceeded the former construction to commutative bialgebroids over $R$. The idea of the present paper is to bring these constructions to the context of 
symmetric finite tensor categories and investigate how far we may get in this setting. 

We introduce bialgebroid categories over a finite tensor category $\C$. In \cite{Femic3} we introduced $\C$-comonoidal categories as coalgebra objects in the 2-category 
$(\C\x\Bimod, \Del_{\C}, \C)$. We define a bialgebroid category $\A$ as a $\C$-bimodule category which is monoidal and $\C$-comonoidal with certain compatibility conditions. 
Intuitively, these conditions require that the commultiplication and the counit functors from the comonoidal structure are monoidal, though some caution is needed, as 
$\A\Del_{\C}\A$ in general is not a monoidal category. To tackle this issue, we present two definitions, one for a non-braided and the other for a braided category $\A$, which 
obviously coincide in the braided setting. 
We concentrate on the case where both $\C$ and $\A$ are symmetric as monoidal categories. 
Then we introduce a new cohomology over $\A$, called Harrison cohomology, which for $\A=\C\Del\C$ coincides with the Amitsur 
cohomology over $\C$ from \cite{Femic2}. In particular, we have that $H^n(\A, P)=H^{n+1}(\C, P)$, where $P$ is a suitable functor from symmetric tensor categories 
to abelian groups (or symmetric monoidal categories). Thus $n$-cocycles over $\A$ are $n+1$-cocycles over $\C$. 
 We record that there exists an analogous infinite exact sequence as in the latter article, which is constructed following {\em mutatis mutandis} 
the ideas from therein. In order to interpret the middle cohomology group in the third level of this sequence, we consider quasi-monoidal structures on the bicategory of representations 
$\A\x\Mod$. These are tensor products on 0-cells with an associativity constraint which satisfies the pentagonal axiom up to a natural equivalence. In general these associativity constraints 
are not identities and this is why we deal indeed with a bicategory, rather than a 2-category $\A\x\Mod$. We call these structures ``quasi-monoidal'' as we are not interested at this point 
in the unit object, nor the corresponding unity constraints and other axioms for a monoidal bicategory structure on $\A\x\Mod$. 

We generate  quasi-monoidal structures on $\A\x\Mod$ using (additive) autoequivalences of $\A\Del_{\C}\A\x\Mod$. For this purpose we develope an Eilenberg-Watts-type Theorem for 2-categories. 
We do this in two stages. Firstly, in \thref{pre-EW} 
we start with the (2-)category $\Pseud(\C\x\Mod, \D\x\Mod)$ of pseudo-functors $\C\x\Mod\to\D\x\Mod$ and construct a functor to the (2-)category $\D\x\C\x\Bimod$. We show how the 
hexagonal coherence diagram for the monoidal structure of a pseudo-functor $\F:\C\x\Mod\to\D\x\Mod$ corresponds to the pentagonal axiom for the right $\C$-module category structure 
of $\F(\C)$, while the pentagonal coherence for a pseudo-natural transformation $\omega:\F\to\G$ of pseudo-functors $\F, \G: \C\x\Mod\to\D\x\Mod$ corresponds to the pentagonal 
axiom for the right $\C$-linearity of the functor $\omega(\C):\F(\C)\to\G(\C)$. The functor $\Pseud(\C\x\Mod, \D\x\Mod)\to\D\x\C\x\Bimod$ induced in this way is cleary essentially surjective, 
though it is not full and fails to be an equivalence of categories. In \thref{EW} we restrict to the category $2\x\Fun_{cont}(\D\x\Mod, \C\x\Mod)$ of 2-functors which preserve 
arbitrary coproducts and cokernels and we prove that it is equivalent to $\D\x\C\x\Bimod$. For a consequence, taking $\D=\C$ and considering invertible cells on both levels, we 
deduce that every autoequivalence of $\C\x\Mod$ is equivalent to $\M\Del_{\C}-$ for some $[\M]\in\BrPic(\C)$, where $\BrPic(\C)$ denotes the Brauer-Picard group 
introduced in \cite{ENO}. 

Finally, we prove that the quasi-monoidal structures on $\A\x\Mod$ that we study form a monoidal category, whose Grothendieck group is isomorphic to the group $Z^2(\A, \dul{\Pic})$ 
of 2-cocycles in our Harrison cohomology with values in $\dul{\Pic}(\A\Del_{\C}\A)$. For $\A=\C\Del\C$ we get an interpretation of the 3-cocycles over $\C$. 
Given a braided finite tensor category $\C$, we denote by $\dul{\Pic}(\C)$ the category of invertible one-sided $\C$-bimodule categories. 
One-sided bimodule categories were studied in \cite{DN}. The Grothendieck group $\Pic(\C)$ of $\dul{\Pic}(\C)$ is a subgroup of 
$\BrPic(\C)$. In \cite{Femic2} we proved that when $\C$ is symmetric, the category $(\dul{\Pic}(\C), \Del_{\C}, \C)$ is symmetric monoidal and thus the group $\Pic(\C)$ is abelian. 
This fact underlies our construction of the Amitsur cohomology over symmetric finite tensor categories in \cite{Femic2} and the Harrison cohomology over symmetric bialgebroid 
categories in the present paper. Due to \cite[Proposition 2.7]{EO} we know that any symmetric finite tensor category $\C$ is equivalent to a category of finite-dimensional representations of a 
finite-dimensional triangular weak quasi-Hopf algebra $H$. Moreover, by \cite[Theorem 5.1.1]{AEG}, \cite[Theorem 4.3]{EG}, if $H$ is a Hopf algebra and the underlying field is 
algebraically closed and of characteristic zero, $H$ is the Drinfel'd twist of a modified supergroup algebra. 

\medskip

In the coming section we set some preliminary notions and results. Here we deduce a relation between the composition of successive morphisms acting between certain invertible objects 
in $\C$, on the one hand-side, and their tensor product, on the other. This will be of big importance in \thref{asoc-coc}. \seref{EW} is dedicated to the Eilenberg-Watts Theorem 
for 2-categories, \seref{Bialgebroids} to the bialgebroid categories and \seref{Harrison} to our version of the Hsrrison cohomology over symmetric bialgebroid categories $\A$. 
In the last section we study certain quasi-monoidal structures in $\A\x\Mod$. In \thref{asoc-coc} we prove how the coherence pentagon for the associativity constraints correspond to the 
2-cocycle condition over $\dul{\Pic}(\A\Del_{\C}\A)$, whereas in \thref{monoidal equiv} we prove that the latter induces a monoidal equivalence between the category of such 
quasi-monoidal structures and of the 2-cocycles. The obtained equivalence of categories finally yields the corresponding group isomorphism.

\section{Preliminaries}

We assume that the reader is familiar with the notions of a finite tensor category, a left/right/bimodule category over a finite tensor category, (bi)module functors,
Deligne tensor product of finite abelian categories, tensor product of bimodule categories. For the basics on the subject we refer to \cite{EO}, \cite{ENO}, \cite{Gr}, \cite{EGNObook}.
All categories will be finite, and $k$-linear and all functors will be $k$-linear, where $k$ throughout will be an algebraically closed field of characteristic zero. 
When there is no confusion we will denote the identity functor on a category $\M$ by $\M$. 

For basic notions on bicategories we refer to \cite{Be, Bo}. The definitions recalled in \cite{Ga}, \cite[Section 1.1]{Neu} are sufficient for the results
in this paper. The corresponding notions of ``functors'' and ``natural transformations'' for bicategories are ``pseudo-functors'' and
``pseudo-natural transformations''. 
When the associativity and unity constraints for 0-cells in a bicategory are identities,
the bicategory is said to be a 2-category. For a finite tensor category $\C$ we have the 2-category $\C\x\Mod$ whose 0-cells are $\C$-module categories,
1-cells are left $\C$-module functors and 2-cells are left $\C$-natural transformations.

A module category is said to be strict if the natural transformations for the associativity and unit constraints for the action 
are identities. A module functor is said to be strict if the natural transformation for the linearity of the functor 
is identity.

\medskip

As we observed in \cite[Section 3]{Femic2}, given a tensor functor $\eta:\C\to\E$ there is a right adjoint functor to $\E\Del_{\C}-: \C\x\Mod\to\E\x\Mod$
which we call restriction of scalars functor $\R: \E\x\Mod\to\C\x\Mod$ defined as follows. Given a left $\E$-module category $\M$ it is a left $\C$-module category
via $\eta$, that is, for $M\in\M, C\in\C$ it is $C\crta\ot M=\eta(C)\crta\ot M$. We will use the following notation:
\begin{equation} \eqlabel{res}
\R(\M)={}_{\eta}\M.
\end{equation}

A right $\ca$-module category $\M$ gives rise to a left $\C$-module category $\mo^{op}$ with the action
given by \equref{left op} and associativity isomorphisms $m^{op}_{X,Y,M}= m_{M,{}^*Y, {}^*X}$ for all $X, Y\in \ca, M\in \mo$.
Similarly, a left $\ca$-module category $\mo$ gives rise to a right $\ca$-module category $\mo^{op}$ with the action 
given via \equref{right op}. Here ${}^*X$ denotes
the left dual object and $X^*$ the right dual object for $X\in\C$. If $\mo$ is a $(\ca,\Do)$-bimodule category then $\mo^{op}$ is a
$(\Do,\ca)$-bimodule category and $(\mo^{op})^{op}\iso\M$ as $(\ca,\Do)$-bimodule categories. \vspace{-0,7cm}
\begin{center}
\begin{tabular}{p{4.8cm}p{1,2cm}p{5.4cm}}
\begin{eqnarray}  \eqlabel{left op}
X\crta\ot^{op}M=M\otb {}^*X
\end{eqnarray}  & &
\begin{eqnarray} \eqlabel{right op}
M\crta\ot^{op}X=X^*\crta\ot M
\end{eqnarray}
\end{tabular}
\end{center} \vspace{-0,7cm}

For a $\C\x\D$-bimodule functor $\F:\M\to\N$ the $\D\x\C$-bimodule functor ${}^{op}\F: {}^{op}\N\to{}^{op}\M$ is given by
${}^{op}\F=\sigma_{\M, \C}^{-1}\comp\F^*\comp\sigma_{\N, \C}$. Here $\F^*: \Fun_{\C}(\N, \C)\to\Fun_{\C}(\M, \C)$ is given by
$\F^*(G)=G\comp\F$ and $\sigma$ is the equivalence proved in \cite[Lemma 4.3]{Femic2} and given by
$\sigma_{\M,\N}: \M\Del_{\C}{}^{op}\N\to\Fun(\N,\M)_{\C}, \sigma(M\Del_{\C}N)=M\crta\ot \ol\Hom_{\N}(-, N)$. The object $\ol\Hom_{\N}(N', N)\in\C$ is
the inner hom-object. If $\F$ is an equivalence, it is $(\F^*)^{-1}=(\F^{-1})^*=-\comp\F^{-1}$ and consequently: $({}^{op}\F)^{-1}={}^{op}(\F^{-1})$.

\medskip

A $(\ca, \Do)$-bimodule category $\mo$ is called \emph{invertible} \cite{ENO} if there are  equivalences of bimodule categories
$$\mo^{op}\boxtimes_{\ca} \mo\simeq \Do, \quad \mo\boxtimes_{\Do} \mo^{op}\simeq \ca.$$
The group of equivalence classes of exact invertible module categories over a finite tensor category $\C$ is called the Brauer-Picard group. It was
introduced in \cite{ENO} and it is denoted by $\BrPic(\C)$.

When $\C$ is braided, then every left $\C$-module category can be made into a right and a $\C$-bimodule category, by putting:
$M\crta\ot X=X\crta\ot M$, \cite[Section 2.8]{DN}. This kind of $\C$-bimodule categories is called {\em one-sided $\C$-bimodule categories.} They form a monoidal
subcategory $(\C^{br}\x\Mod, \Del_{\C}, \C)$ of $(\C\x\dul\Mod, \Del_{\C}, \C)$. As we proved in \cite[Proposition 3.4]{Femic2} when $\C$ is symmetric, $(\C^{br}\x\Mod, \Del_{\C}, \C)$ is symmetric.

For a braided finite tensor category $\C$ we will denote by $\dul\Pic(\C)$ the monoidal category $(\dul\Pic(\C), \Del_{\C}, \C)$ of exact invertible
one-sided $\C$-bimodule categories. The Grothendieck group of $\dul\Pic(\C)$, that is, the Picard group of equivalence classes of exact invertible
one-sided $\C$-bimodule categories, is denoted by $\Pic(\C)$. It is a subgroup of $\BrPic(\C)$, \cite[Section 4.4]{ENO}, \cite[Section 2.8]{DN}. In view of the
above said, when $\C$ is symmetric, the group $\Pic(\C)$ is clearly abelian.

\medskip

In \cite[Lemma 4.9]{Femic2} and \cite[Lemma 2.7]{Femic3} respectively we proved:

\begin{lma}\lelabel{basic lema}
Let 
$\F, \G: \M\to\N$ be $\D\x\C$-bimodule functors and let $\Pp$ be an invertible $\C$-bimodule category.
\begin{enumerate}
\item It is $\F\Del_{\C}\Id_{\Pp}=\G\Del_{\C}\Id_{\Pp}$ if and only if $\F=\G$.
\item If $\HH: \Pp\to\Ll$ is a $\C$-bimodule equivalence functor, it is $\F\Del_{\C}\HH=\G\Del_{\C}\HH$ if and only if $\F=\G$.
\end{enumerate}
\end{lma}

\begin{lma} \lelabel{alfa dagger}
Let $\C$ be a finite tensor category and $\M$ an exact $\C$-bimodule category and let $\alpha: \M\to\C$ be a $\C$-bimodule equivalence. Define
$\alpha^{\dagger}=\crta\ev(id_{{}^{op}\M}\Del_{\C}\alpha^{-1}): {}^{op}\M\to\C$.
Then:
\begin{enumerate}
\item $\crta\ev=\alpha^{\dagger} \Del_{\C}\alpha$ as right $\C$-linear functors.
\item $\alpha^{\dagger}=({}^{op}\alpha)^{-1}$ as $\C$-bimodule functors.
\end{enumerate}
\end{lma}

\subsection{Composition versus tensor product of morphisms in $\dul{\Pic}(\C)$ for $\C$ symmetric} \sslabel{1st subsection} 

In the developing paper ``Harrison cohomology over commutative Hopf algebroids'' by Stefaan Caenepeel and the author we proved some
results on symmetric Picard groupoids that will be useful here. In what follows we prove them in a different form.

\medskip

Let $\C$ be a braided monoidal category with braiding $\Phi$ and unit object $I$. Let $X^*$ be a right dual object for $X\in\C$ with morphisms 
$\crta{ev}: X^*\ot X\to I$ and $\crta{coev}: I\to X\ot X^*$. We will use the following notation:  
$$\crta{ev}=
\gbeg{3}{1}
\got{1}{X^*} \got{1}{X} \gnl
\gev \gnl
\gob{1}{}
\gend
\qquad\textnormal{and}\qquad
\crta{coev}=
\gbeg{3}{3}
\got{1}{} \gnl
\gdb \gnl
\gob{1}{X} \gob{2}{\hspace{-0,2cm}X^*.}
\gend 
$$
Then ${}^*X=X^*$ with morphisms
\begin{equation} \eqlabel{coev new}
ev=\crta{ev}\comp\Phi_{X, X^*}: X\ot {}^*X\to I \qquad \textnormal{and} \qquad
coev=\Phi^{-1}_{X, X^*}\comp\crta{coev}: I\to {}^*X\ot X
\end{equation}
is a left dual object for $X$ in $\C$. Recall that invertible objects in $\C$ are such objects $X\in\C$ for which there exists $X'\in\C$ so that $X\ot X'\iso I$.

\begin{lma} \lelabel{special}
The following are equivalent:
\begin{enumerate}
\item an invertible object $X\in\C$ has a right dual and it is $\crta{coev}\comp\crta{ev}=\Phi_{X^*, X}$;
\item for an invertible object $X\in\C$ it is $\Phi_{X, X}=\Id_{X\ot X}$ (equivalently: $\Phi_{X, X}^{-1}=\Id_{X\ot X}$).
\end{enumerate}
\end{lma}

\begin{proof}
The condition (1) implies that
\begin{equation}\eqlabel{cuenta}
\scalebox{0.9}[0.9]{
\gbeg{2}{5}
\got{1}{X^*} \got{1}{X} \gnl
\gev  \gnl
\gdb \gnl
\gibr \gnl
\gob{1}{X^*} \gob{1}{X}
\gend}=
\scalebox{0.9}[0.9]{
\gbeg{3}{5}
\got{1}{X^*} \got{1}{X} \gnl
\gcl{3} \gcl{3} \gnl
\gob{1}{X^*} \gob{1}{X}
\gend}\qquad\textnormal{hence:}\qquad
\scalebox{0.9}[0.9]{
\gbeg{4}{3}
\got{2}{} \got{1}{X} \got{1}{X} \gnl
\gdb \gcl{1} \gcl{1} \gnl
\gob{1}{X} \gob{1}{X^*} \gob{1}{X} \gob{1}{X}
\gend}=
\scalebox{0.9}[0.9]{
\gbeg{4}{6}
\got{2}{} \got{1}{X} \got{1}{X} \gnl
\gdb \gcl{1} \gcl{4} \gnl
\gcl{3} \gev \gnl
\gvac{1} \gdb \gnl
\gvac{1} \gibr \gnl
\gob{1}{X} \gob{1}{X^*} \gob{1}{X} \gob{1}{X}
\gend}=
\scalebox{0.9}[0.9]{
\gbeg{4}{4}
\got{1}{X} \got{2}{} \got{1}{X} \gnl
\gcl{2} \gdb \gcl{2} \gnl
\gvac{1} \gibr \gnl
\gob{1}{X} \gob{1}{X^*} \gob{1}{X} \gob{1}{X}
\gend}
\end{equation}
Then we have:
$$
\scalebox{0.9}[0.9]{
\gbeg{2}{3}
\got{1}{X} \got{1}{X} \gnl
\gbr \gnl
\gob{1}{X} \gob{1}{X}
\gend}=
\scalebox{0.9}[0.9]{
\gbeg{4}{4}
\got{2}{} \got{1}{X} \got{1}{X} \gnl
\gdb \gibr \gnl
\gcl{1} \gev \gcl{1} \gnl
\gob{1}{X} \gob{5}{X}
\gend}\stackrel{\equref{cuenta}}{=}
\scalebox{0.9}[0.9]{
\gbeg{4}{6}
\got{1}{X} \got{2}{} \got{1}{X} \gnl
\gcl{4} \gdb \gcl{2} \gnl
\gvac{1} \gibr \gnl
\gvac{1} \gcl{1} \gibr \gnl
\gvac{1} \gev \gcl{1} \gnl
\gob{1}{X} \gob{5}{X}
\gend}=
\scalebox{0.9}[0.9]{
\gbeg{3}{4}
\got{1}{X} \got{1}{X} \gnl
\gcl{2} \gcl{2} \gnl
\gob{1}{X} \gob{1}{X}
\gend}
$$
Suppose that (2) holds and fix an isomorphism $e: X'\ot X\to I$ for every invertible object $X\in\C$. Define $c=\Phi_{X', X}\comp e^{-1}$, so we have:
$$(X\ot e)(c\ot X)=
\scalebox{0.9}[0.9]{
\gbeg{4}{5}
\got{2}{} \got{1}{X} \gnl
\glmpb \gnot{\hspace{-0,2cm}e^{-1}} \grmpb \gcl{2} \gnl
\gbr \gnl
\gcl{1} \glmpt\gnot{\hspace{-0,2cm}e}\grmpt \gnl
\gob{1}{X}
\gend}\stackrel{(2)}{=}
\scalebox{0.9}[0.9]{
\gbeg{4}{6}
\got{2}{} \got{1}{X} \gnl
\glmpb \gnot{\hspace{-0,2cm}e^{-1}} \grmpb \gcl{1} \gnl
\gcl{1} \gbr \gnl
\gbr \gcl{1} \gnl
\gcl{1} \glmpt\gnot{\hspace{-0,2cm}e}\grmpt \gnl
\gob{1}{X}
\gend}=
\scalebox{0.9}[0.9]{
\gbeg{3}{4}
\got{1}{X} \gnl
\gcl{2} \gnl
\gob{1}{X}
\gend}
$$
thus one law for right duality is satisfied, and similarly the other one is proved.
\qed\end{proof}

For invertible objects in $\C$ satisfying one of the equivalent conditions of the above lemma we will say that {\em they satisfy the property P}. 
The notation $\Phi^{\pm 1}$ will mean that the statement in question is valid both for $\Phi$ and $\Phi^{-1}$.

\begin{cor} \colabel{twist two morfs}
Given morphisms $f:X\to Y$ and $g:X\to Z$ between invertible objects $X,Y,Z$ 
satisfying the property P, then
\begin{equation} \eqlabel{braid}
\Phi_{Y,Z}^{\pm 1}(f\ot g)=g\ot f.
\end{equation}
If $Y=Z$, then $f\ot g=g\ot f$.
\end{cor}

\begin{proof}
By naturality (both of $\Phi$ and $\Phi^{-1}$) we have: $\Phi_{Y,Z}(f\ot g)=(g\ot f)\Phi_{X,X}\stackrel{\leref{special}}{=}g\ot f$. The rest also follows by \leref{special}.
\qed\end{proof}

In the following three claims, which are consequences of \coref{twist two morfs}, in all the appearances of the braiding $\Phi$ in the assertions the analogous statements with
$\Phi$ replaced by $\Phi^{-1}$ are valid.

Observe that for invertible objects $X$ satisfying the property P it is
\begin{equation}\eqlabel{odnos coev}
\crta{ev}\comp coev=\crta{ev}\comp\Phi^{-1}\comp\crta{coev}=\Id_I.
\end{equation}


\begin{prop}
Given morphisms $f:X\to Y, g:Y\to Z$ and $h:X\to Z$ between invertible objects $X,Y,Z$ in $\C$ satisfying the property P, observe that
$f\ot g, (Y\ot h)\Phi_{X,Y}=\Phi_{Z,Y}(h\ot Y): X\ot Y\to Y\ot Z$. Then it is
$$f\ot g=(Y\ot h)\Phi_{X,Y}\qquad\textnormal{if and only if}\qquad h=g\comp f.$$
\end{prop}

\begin{proof}
It is:
$$\gbeg{2}{5}
\got{1}{X} \gnl
\gcl{1} \gnl
\gbmp{h} \gnl
\gcl{1} \gnl
\gob{1}{Z}
\gend\stackrel{\equref{odnos coev}}{=}
\gbeg{4}{5}
\got{5}{X} \gnl
\gdb \gcl{1} \gnl
\gibr \gbmp{h} \gnl
\gev \gcl{1} \gnl
\gob{5}{Z}
\gend=
\gbeg{4}{7}
\got{5}{X} \gnl
\gdb \gcl{2} \gnl
\gibr \gnl
\gcl{1} \gibr \gnl
\gcl{1} \gbr \gnl
\gev \gbmp{h} \gnl
\gob{5}{Z}
\gend\stackrel{*}{=}
\gbeg{4}{7}
\got{5}{X} \gnl
\gdb \gcl{2} \gnl
\gibr \gnl
\gcl{1} \gibr \gnl
\gcl{1} \gbmp{f} \gbmp{g} \gnl
\gev \gcl{1} \gnl
\gob{5}{Z}
\gend\stackrel{nat.}{=}
\gbeg{4}{7}
\got{5}{X} \gnl
\gdb \gcl{2} \gnl
\gibr \gnl
\gcl{1} \gbmp{g} \gbmp{f} \gnl
\gcl{1} \gibr \gnl
\gev \gcl{1} \gnl
\gob{5}{Z}
\gend\stackrel{\equref{braid}}{=}
\gbeg{4}{7}
\got{5}{X} \gnl
\gdb \gcl{2} \gnl
\gibr \gnl
\gcl{2} \gcl{2} \gbmp{f} \gnl
\gvac{2} \gbmp{g} \gnl
\gev \gcl{1} \gnl
\gob{5}{Z}
\gend\stackrel{\equref{odnos coev}}{=}gf
$$
where we used the assumption $f\ot g=(Y\ot h)\Phi_{X,Y}$ at the place *. Conversely, $f\ot g=(Y\ot g)(f\ot Y)\stackrel{\equref{braid}}{=}\Phi_{Z,Y}\comp(g\ot Y)(f\ot Y)
\stackrel{*}{=}\Phi_{Z,Y}\comp(h\ot Y)\stackrel{nat.}{=}(Y\ot h)\comp\Phi_{X,Y}$ where * represents the assumption $h=gf$.
\qed\end{proof}

\begin{cor}
Given morphisms $f:X\to Y$ and $g:Y\to Z$ between invertible objects $X,Y,Z$ in a braided monoidal category $\C$ satisfying the property P, it is:
\begin{equation} \eqlabel{braid rule}
\gbeg{3}{5}
\got{1}{X} \got{1}{Y} \gnl
\gbr \gnl
\gcl{1} \gbmp{f} \gnl
\gcl{1} \gbmp{g} \gnl
\gob{1}{Y} \gob{1}{Z}
\gend=
\gbeg{3}{5}
\got{1}{X} \got{1}{Y} \gnl
\gcl{1} \gcl{1} \gnl
\gbmp{f} \gbmp{g} \gnl
\gcl{1} \gcl{1} \gnl
\gob{1}{Y} \gob{1}{Z}
\gend
\end{equation}
\end{cor}

\begin{cor} \colabel{comp-tensor}
Given a chain of morphisms $D_0\stackrel{f_1}{\to}D_1\stackrel{f_2}{\to} \dots \stackrel{f_n}{\to}D_n$ between invertible objects
$D_0,D_1,\dots, D_n$ in a braided monoidal category $\C$ satisfying the property P. Then it is:
\begin{equation} \eqlabel{comp-tensor}
(D_1\ot D_2\ot\dots\ot D_{n-1}\ot(f_n\comp f_{n-1}\comp\dots\comp f_1))\comp\Phi_{D_0,D_1\ot D_2\ot\dots\ot D_{n-1}}=f_1\ot f_2\ot\dots\ot f_n.
\end{equation}
\end{cor}

\begin{proof}
We have:
$$
\gbeg{7}{5}
\got{1}{D_0} \got{1}{D_1} \got{1}{D_2} \got{2}{\dots} \got{1}{D_n} \gnl
\gcl{1} \gcl{1} \gcl{1} \gcl{1}  \gcn{1}{1}{0}{0} \gcn{1}{1}{-1}{-1}  \gcn{1}{1}{-1}{-1}  \gnl
\gbmp{f_1} \gbmp{f_2} \gbmp{f_3}  \glmptb\gnot{\hspace{-0,4cm}\dots}\grmptb  \gbmp{f_n} \gnl
\gcl{1} \gcl{1} \gcl{1} \gcl{1} \gcn{1}{1}{0}{0} \gcn{1}{1}{-1}{-1}  \gcn{1}{1}{-1}{-1} \gnl
\gob{1}{D_1} \gob{1}{D_2} \gob{1}{D_3} \gob{2}{\dots} \gob{1}{D_n}
\gend\stackrel{\equref{braid rule}}{=}
\gbeg{7}{5}
\got{1}{D_0} \got{1}{D_1} \got{1}{D_2} \got{2}{\dots} \got{1}{D_n} \gnl
\gbr \gcl{1} \gcl{1}  \gcn{1}{1}{0}{0} \gcn{1}{1}{-1}{-1}  \gcn{1}{1}{-1}{-1} \gnl
\gcl{1} \gbmp{f_1} \gbmp{f_3} \glmptb\gnot{\hspace{-0,4cm}\dots}\grmptb \gbmp{f_n} \gnl
\gcl{1} \gbmp{f_2} \gcl{1} \gcl{1} \gcn{1}{1}{0}{0} \gcn{1}{1}{-1}{-1}  \gcn{1}{1}{-1}{-1} \gnl
\gob{1}{D_1} \gob{1}{D_2} \gob{1}{D_3} \gob{2}{\dots} \gob{1}{D_n}
\gend\stackrel{\equref{braid rule}}{=}
\gbeg{7}{7}
\got{1}{D_0} \got{1}{D_1} \got{1}{D_2} \got{2}{\dots} \got{1}{D_n} \gnl
\gbr \gcl{1} \gcl{1}  \gcn{1}{1}{0}{0} \gcn{1}{1}{-1}{-1}  \gcn{1}{2}{-1}{-1} \gnl
\gcl{1} \gbr \glmptb\gnot{\hspace{-0,4cm}\dots}\grmptb  \gnl
\gcl{1} \gcl{1} \gbmp{f_1} \gcl{1} \gcn{1}{1}{0}{0} \gcn{1}{1}{-1}{-1} \hspace{-0,42cm} \gbmp{f_n} \gnl
\gvac{1} \gcl{1} \gcl{1} \gbmp{f_2} \gcl{1} \gcn{1}{1}{0}{0} \gcn{1}{1}{-1}{-1}  \gcn{1}{1}{-1}{-1} \gnl
\gvac{1} \gcl{1} \gcl{1} \gbmp{f_3} \gcl{1} \gcn{1}{1}{0}{0} \gcn{1}{1}{-1}{-1}  \gcn{1}{1}{-1}{-1} \gnl
\gvac{1} \gob{1}{D_1} \gob{1}{D_2} \gob{1}{D_3} \gob{2}{\dots} \gob{1}{D_n}
\gend$$

$$=\dots=
\gbeg{9}{8}
\got{1}{D_0} \got{1}{D_1} \got{1}{D_2} \got{2}{\dots} \gvac{2} \got{1}{D_n} \gnl
\gbr \gcl{1} \gcl{2}  \gcn{1}{2}{0}{0} \gcn{1}{2}{-1}{-1} \gcn{1}{2}{-1}{-1}  \gcl{3} \gnl
\gcl{1} \gbr \gnl
\gcl{1} \gcl{1} \gvac{1} \glmptb\gnot{\hspace{-0,4cm}\dots}\grmptb \gnl
\gcl{3} \gcl{3} \gcl{3} \gcl{3} \gcn{1}{3}{0}{0} \gcn{1}{3}{-1}{-1} \hspace{-0,42cm} \gbmp{f_1} \gvac{1} \gbmp{f_n} \gnl
\gvac{6} \gbmp{\dots} \gvac{1} \gcl{2} \gnl
\gvac{6} \glmptb\gnot{\hspace{-0,4cm}f_{n-1}}\grmp \gnl
\gvac{1} \gob{1}{D_1} \gob{1}{D_2} \gob{1}{D_3} \gob{2}{\dots} \gob{2}{D_{n-1}} \gob{1}{D_n}   
\gend\stackrel{\equref{braid rule}}{=}
\gbeg{8}{9}
\got{1}{D_0} \got{1}{D_1} \got{1}{D_2} \got{2}{\dots} \got{1}{\hspace{-0,34cm} D_{n-1}} \got{1}{D_n} \gnl
\gbr \gcl{1} \gcl{2} \gcn{1}{2}{0}{0} \gcn{1}{2}{-1}{-1} \gcn{1}{2}{-1}{-1} \gcn{1}{3}{-1}{-1}  \gnl
\gcl{1} \gbr \gnl
\gcl{1} \gcl{1} \gvac{1} \glmptb\gnot{\hspace{-0,4cm}\dots}\grmptb \gnl
\gcl{4} \gcl{4} \gcl{4} \gcl{4} \gcn{1}{4}{0}{0} \gcn{1}{4}{-1}{-1} \hspace{-0,42cm} \gbr \gnl
\gvac{6} \gcl{3} \gbmp{f_1} \gnl
\gvac{7} \gbmp{\dots} \gnl
\gvac{7} \gbmp{f_n} \gnl
\gvac{1} \gob{1}{D_1} \gob{1}{D_2} \gob{1}{D_3} \gob{2}{\dots} \gob{1}{D_{n-1}} \gob{2}{D_n}
\gend
$$
in the last two diagrams two obvious lines are missing because they are involved in the chain of the braidings in the middle strings. The claim
follows by the braiding axioms.
\qed\end{proof}

In \cite[Section 4]{Femic2} we proved that in a symmetric tensor category $\C$ for an object $\M\in\dul{\Pic}(\C)$ its (right and left) dual object in $(\C\x\Bimod, \Del_{\C}, \C)$
is $\M^{op}$ and that the corresponding evaluation and coevaluation functors are equivalences. Moreover, in \cite[Corollary 4.11]{Femic2} we showed: $coev\simeq ev^{-1}$
up to the symmetry $\tau: \M\Del_{\C}\N\stackrel{\simeq}{\to}\N\Del_{\C}\M$ induced by $M\Del N\mapsto N\Del_{\C}M$, for $\M, \N\in\dul{\Pic}(\C)$. This means that
every $\M\in\dul{\Pic}(\C)$ satisfies the property $P$.
Therefore, given 
a chain of morphisms $\M_0\stackrel{\F_1}{\to}\M_1\stackrel{\F_2}{\to} \dots \stackrel{\F_n}{\to}\M_n$ in $\dul{\Pic}(\C)$ one has:
$\F_1\Del_{\C}\dots\Del_{\C}\F_n\simeq(\M_1\Del_{\C}\M_2\Del_{\C}\dots\Del_{\C}\M_{n-1}\Del_{\C}(\F_n\comp\F_{n-1}\comp\dots\comp\F_1))\comp\hspace{0,2cm}
\tau_{\M_0,\M_1\Del_{\C}\M_2\Del_{\C}\dots\Del_{\C}\M_{n-1}}$.

\section{Eilenberg-Watts Theorem for 2-categories: module categories over finite tensor categories} \selabel{EW}

A version of Eilenberg-Watts Theorem (proved in \cite{Watts, Eilenberg}) is valid also in the context of module categories over finite tensor categories.
Given a finite tensor category $\C$, any additive equivalence 2-endofunctor  $F$ of $\C\x\Mod$ is given by $F\iso\M\Del_{\C}-$ for an invertible $\C$-bimodule category $\M$.
We prove this result here.

Let $\Pseud(\C\x\Mod,\D\x\Mod)$ denote the category of pseudo-functors $\C\x\Mod\to\D\x\Mod$ and pseudo-natural transformations between them. It is actually a 2-category, where
2-cells are modifications between pseudo-natural transformations, as so are $\C\x\Mod$ and $\D\x\Mod$, but we consider them as ordinary categories here.

\begin{thm} \thlabel{pre-EW}
Let $\C$ and $\D$ be finite tensor categories. There is a functor
$$\Omega: \Pseud(\C\x\Mod,\D\x\Mod)\to\D\x\C\x\Bimod.$$
\end{thm}

\begin{proof}
\dul{On 0-cells.}
Let $\F:\C\x\Mod\to\D\x\Mod$ be a pseudo-functor. For $\N\in\C\x\Mod$ we have a functor:
\begin{equation} \eqlabel{func C,N}
\F_{\C, \N}: \N\simeq\Fun_{\C}(\C, \N)\to\Fun_{\D}(\F(\C), \F(\N)).
\end{equation}
Given two functors 
$G\in\Fun_{\C}(\C, \C)$ and $H\in\Fun_{\C}(\C, \N)$ we have:
\begin{equation} \eqlabel{func comp}
\theta^{\C, \N}_{H,G}: \F_{\C, \N}(H\comp G)\stackrel{\iso}{\to}\F_{\C, \N}(H)\F_{\C, \C}(G)\quad\textnormal{and}\quad\F_{\C, \C}(\Id_{\C})\simeq\Id_{\F_{\C, \C}(\C)}
\end{equation}
where $\theta^{\C, \N}_{H,G}$ are isomorphisms natural in $\C$ and $\N$. Let $N\in\N$ and define $F_N=-\crta\ot N\in\Fun_{\C}(\C, \N)$. Given $C\in\C$ we clearly have:
$F_{C\crta\ot N}\iso F_N\comp F_C$, hence we identify: $\F_{\C, \N}(F_{C\crta\ot N})=\F_{\C, \N}(F_N\comp F_C)$. For this reason by \equref{func comp} we have a natural isomorphism
\begin{equation} \eqlabel{left C-lin fun}
\theta^{\C, \N}_{F_N,F_C}: \F_{\C, \N}(F_{C\crta\ot N}) \stackrel{\iso}{\to} \F_{\C, \N}(F_N)\F_{\C, \C}(F_C).
\end{equation}
For $\N=\C$ the functor $\F_{\C, \C}: \C\to\Fun_{\D}(\F(\C), \F(\C))$ makes $\M:=\F(\C)$ a right $\C$-module category. Given $M\in\M$ define
\begin{equation} \eqlabel{M right C-mod}
M\crta\ot C=\F_{\C, \C}(C)(M).
\end{equation}
Take another $D\in\C$, then we define the associativity constraint for the right $\C$-action by
\begin{equation} \eqlabel{asoc a}
a_{M, C,D}=\theta^{\C, \C}_{F_D,F_C}(M): \F_{\C, \C}(F_{C\ot D})(M)\stackrel{\iso}{\to}\F_{\C, \C}(F_D)\F_{\C, \C}(F_C)(M).
\end{equation}
Now the pentagonal coherence for $\M$ to be a right $\C$-module category holds because of the hexagonal coherence for $\theta^{\C, \C}_{F_C,F_D}(M)$:
\begin{equation} \eqlabel{hexagon theta}
\end{equation} \vspace{0,2cm}
$$
\scalebox{0.84}{
\bfig
\putmorphism(-200,500)(1,0)[\F(F_C\crta\comp(F_D\crta\comp F_E)) ` \F(F_C)\crta\comp\F(F_D\crta\comp F_E) ` \theta_{F_C, F_D\crta\comp F_E}]{1400}1a
\putmorphism(1150,500)(1,0)[\phantom{(X \ot (Y \ot U)) \ot W}` \F(F_C)\crta\comp(\F(F_D)\crta\comp\F(F_E))  ` \Id_{\F(F_C)}\crta\comp\theta_{F_D, F_E}]{1650}1a
\putmorphism(2800,500)(0,-1)[`(\F(F_C)\crta\comp\F(F_D))\crta\comp\F(F_E) `\alpha']{500}1l
\putmorphism(-160,500)(0,-1)[``\F(\alpha)]{500}1r
\putmorphism(-200,0)(3,-1)[\F((F_C\crta\comp F_D)\crta\comp F_E) `  ` ]{1160}1l
  \putmorphism(100,-80)(4,-1)[` \F(F_C\crta\comp F_D)\crta\comp\F(F_E) ` ]{1160}0l
\putmorphism(0,-100)(3,-1)[`  ` \theta_{F_C\crta\comp F_D,F_E}]{1160}0l
\putmorphism(1900,-240)(3,1)[`  ` ]{280}1r
\putmorphism(1750,-330)(3,1)[`  ` \theta_{F_C,F_D}\crta\comp\Id_{\F(F_E)}]{500}0r
\efig}
$$
where $E\in\C$ and $\alpha$ and $\alpha'$ are associativity constraints for the functors in $\Fun_{\C}(\C, \C)$. The equation in $\Fun_{\D}(\F(\C), \F(\C))$ encoded in this diagram:
$$(\theta_{F_C,F_D}\crta\comp\Id_{\F(F_E)}) \theta_{F_C\crta\comp F_D,F_E} \F(\alpha)=\alpha' (\Id_{\F(F_C)}\crta\comp\theta_{F_D, F_E})\theta_{F_C, F_D\crta\comp F_E}$$
applied to an object $M\in\M$ and by \equref{asoc a} transforms into:
$$a_{M\crta\ot E,D,C}\comp a_{M,E,D\ot C}\comp(M\crta\ot\mathfrak{a})=(a_{M, E,D}\crta\ot C)a_{M, E\ot D,C}$$  
where $\mathfrak{a}$ is the associativity constraint in $\C$ (the associativity $\alpha'$ becomes reduntant). The rule for the action of the unit is proved similarly, thus
$\M$ is indeed a right $\C$-module category. Since $\F_{\C, \C}(C)$ is left $\D$-linear for every $C\in\C$, the category $\M$ is actually a $\D\x\C$-bimodule category
with the natural isomorphism $\gamma_{D,M,C}: (D\crta\ot M)\crta\ot C\to D\crta\ot (M\crta\ot C)$ given by the following composition:
$$\gamma_{D,M,C}: (D\crta\ot M)\crta\ot C=\F_{\C,\C}(C)(D\crta\ot M)\stackrel{\tilde s_{D,M}}{\iso}D\crta\ot \F_{\C,\C}(C)(M)=D\crta\ot(M\crta\ot C)$$
for every $D\in\D$, where $\tilde s_{D,M}$ is the left $\D$-module functor structure of $\F_{\C,\C}(C)$.

\bigskip

\dul{On 1-cells.}
Given a pseudo-natural transformation $\omega:\F\to\G$ of pseudo-functors $\F, \G:\C\x\Mod\to\D\x\Mod$.
Set $\F(\C)=\M$ and $\G(\C)=\N$. Then $\omega(\C): \M\to\N$ is a left $\D$-module functor. Let us show that it is right $\C$-linear.
Let $\Ll\in\C\x\Mod$ and fix an object $L\in\Ll$. For the left $\C$-module functor $F_L=-\crta\ot L:\C\to\Ll$ the pseudo-naturality of $\omega$ implies that
the diagram
\begin{equation*}
\scalebox{0.88}{\bfig
\putmorphism(-180,400)(1,0)[\M` \N `\omega(\C)]{1800}1a
\putmorphism(-200,0)(1,0)[\F(\Ll)` \G(\Ll),` \omega(\Ll)]{1850}1a
\putmorphism(-200,400)(0,-1)[\phantom{B}``\F_{\C, \Ll}(F_L)]{380}1l
\putmorphism(1620,400)(0,-1)[\phantom{B\ot B}``\G_{\C, \Ll}(F_L)]{380}1r
\efig}
\end{equation*}
commutes up to an isomorphism, that is, there is a natural isomorphism:
\begin{equation} \eqlabel{omega F_L}
\tilde\omega(F_L): \omega(\Ll)\F_{\C, \Ll}(F_L)\stackrel{\iso}{\to}\G_{\C, \Ll}(F_L)\omega(\C).
\end{equation}
Set $\Ll=\C$ and $L=C$, then applying $\tilde\omega(F_C)$ to some $M\in\M$ 
we get that there are isomorphisms natural in $M$ and $C$:
$$\tilde\omega(F_C)(M): \omega(\C)\F_{\C, \C}(F_C)(M)\stackrel{\iso}{\to}\G_{\C, \C}(F_C)\omega(\C)(M).$$
Recall from \equref{M right C-mod} the right $\C$-module category structure of $\M=\F(\C)$ and identify $F_C\in\Fun_{\C}(\C, \C)\simeq\C$ with $C\in\C$.
We get:
\begin{equation} \eqlabel{s MC}
s_{M,C}=\tilde\omega(F_C)(M): \omega(\C)(M\crta\ot C)\stackrel{\iso}{\to}\omega(\C)(M)\crta\ot C.
\end{equation}
Let $D\in\C$ and set $\alpha=\omega(\C):\M\to\N$. Observe that the octogonal coherence for $\omega$ contains three arrows with associativity for 1-cells, 
which become identities when applied to the 2-categories $\C\x\Mod$ and $\D\x\Mod$, so we indeed have a pentagon. Now this pentagonal coherence for $\omega$ 
applied to functors $F_C$ and $F_D$ reads (we read the transformations from the right to the left!):
\begin{equation*}
\scalebox{0.84}{
\bfig
\putmorphism(-200,500)(1,0)[\alpha\crta\comp\F(F_D)\crta\comp\F(F_C)` \alpha\crta\comp\F(F_DF_C) ` \Id\crta\comp\hspace{0,12cm}\theta_{F_D,F_C}]{1560}{-1}a
\putmorphism(1170,500)(1,0)[\phantom{(X \ot (Y \ot U)) \ot W}` \G(F_DF_C)\crta\comp\alpha ` \tilde\omega(F_DF_C)]{1480}1a
\putmorphism(2570,500)(0,-1)[``\theta'_{F_D,F_C}\crta\comp\Id]{500}1l
\putmorphism(-160,500)(0,-1)[``\tilde\omega(F_D)\crta\comp\Id]{500}1r
\putmorphism(-200,0)(1,0)[\G(F_D)\crta\comp\hspace{0,12cm}\alpha\hspace{0,12cm}\crta\comp\F(F_C) ` \G(F_D)\crta\comp\G(F_C)\crta\comp\alpha ` \Id\crta\comp\hspace{0,12cm}\tilde\omega(F_C)]{2760}1b
\efig}
\end{equation*}
Applying the equation $(\Id\crta\comp\hspace{0,12cm}\tilde\omega(F_C))(\tilde\omega(F_D)\crta\comp\Id)(\Id\crta\comp\hspace{0,12cm}\theta_{F_D,F_C})=(\theta'_{F_D,F_C}\crta\comp\Id)\tilde\omega(F_DF_C)$
encoded in this diagram to an object $M\in\M$ and using \equref{s MC} and \equref{asoc a} we get:
$$(s_{M,C}\crta\ot D)s_{M\crta\ot C, D}\comp \alpha(a_{M,C,D})=a_{\alpha(M),C,D}\comp s_{M,C\ot D}.$$
The compatibility with unit is proved similarly and $\omega(\C)$ is a functor in $\D\x\C\x\Bimod$.

\qed\end{proof}

\begin{lma} \lelabel{left C-lin fun}
The functor $\F_{\C, \N}: \N\to\Fun_{\D}(\F(\C), \F(\N))$ from \equref{func C,N} is left $\C$-linear.
\end{lma}

\begin{proof}
Recall that $\Fun_{\D}(\F(\C), \F(\N))$ is a left $\C$-module category by $C\crta\ot F=F(-\crta\ot C)$ for any $F\in\Fun_{\D}(\F(\C), \F(\N))$ and $C\in\C$, and here we
consider the right $\C$-module structure on $\M=\F(\C)$ induced by $\F_{\C, \C}$ via \equref{M right C-mod}. Then taking $F=\F_{\C, \N}(N)$ we have: $C\crta\ot\F_{\C, \N}(N)=
\F_{\C, \N}(N)\F_{\C, \C}(C)$. Now define 
\begin{equation} \eqlabel{lin fun}
s_{C,N}=\theta^{\C, \N}_{F_N,F_C}: \F_{\C, \N}(C\crta\ot N)\stackrel{\iso}{\to}C\crta\ot\F_{\C, \N}(N)
\end{equation}
where we identified $\F_{\C, \N}(C\crta\ot N)=\F_{\C, \N}(F_N\comp F_C)$ and $C\crta\ot\F_{\C, \N}(N)=\F_{\C, \N}(N)\F_{\C, \C}(C)$ as commented above. 
The pentagonal coherence for $(\F_{\C, \N}, s_{C,N})$ to be left $\C$-linear holds, because of the hexagonal coherence for $\theta^{\C, \N}_{F_N,F_C}$:
first replace $F_C, F_D, F_E$ by $F_N, F_C, F_D$ in \equref{hexagon theta} respectively to obtain:
\begin{equation} \eqlabel{translating left}
(\theta_{F_N,F_C}\crta\comp\Id_{\F_{\C, \C}(F_D)}) \theta_{F_N\crta\comp F_C,F_D} \F_{\C, \N}(\alpha)=\alpha' (\Id_{\F_{\C, \N}(F_N)}\crta\comp\theta_{F_C, F_D})\theta_{F_N, F_C\crta\comp F_D}.
\end{equation}
In \equref{asoc a} we gave the right $\C$-module associativity constraint, its left hand-side version would read: $a_{D, C, -}=\theta^{\C, \C}_{F_D,F_C}(-)$. 
Now apply the equation \equref{translating left} to $M\in\M$ to get:
$$(D\crta\ot s_{C,N}) s_{D,C\crta\ot N} \F_{\C, \N}(a_{D,C,N})=a_{D,C,\F_{\C, \N}(N)}s_{D\ot C, N}.$$
The compatibility with unit is proved similarly.
\qed\end{proof}

The functor $\Omega: \Pseud(\C\x\Mod,\D\x\Mod)\to\D\x\C\x\Bimod$ from \thref{pre-EW}, although essentially surjective, is not faithful, and thus it is not an equivalence of categories.
The candidate for the quasi-inverse functor is the one presented in \thref{EW}. Namely, for a $\D\x\C$-bimodule category $\M$ we have a 2-functor $\M\Del_{\C}-: \C\x\Mod\to\D\x\Mod$
and $\Omega(\M\Del_{\C}-)=\M\Del_{\C}\C\simeq\C$. However, for two pseudo-functors $\F, \G$ and a functor $H:\F(\C)\to\G(\C)$ in $\D\x\C\x\Bimod$, we find that
$H\Del_{\C}-: \F(\C)\Del_{\C}-\to\G(\C)\Del_{\C}-$ defines a 2-natural transformation, but there may be a proper pseudo-natural transformation $\alpha: \F(\C)\Del_{\C}-\to\G(\C)\Del_{\C}-$
so that $\Omega(\alpha)=\alpha(\C)$ coincides on objects with $H$, i.e. that there is an isomorphism $\omega(M): \alpha(\C)(M)\to H(M)$ natural in $M\in\F(\C)$.

\medskip

We do have the equivalence with the category $2\x\Fun_{cont}(\C\x\Mod,\D\x\Mod)$ of 2-functors $\C\x\Mod\to\D\x\Mod$ which preserve cokernels and arbitrary coproducts,
and the corresponding 2-natural transformations.
If $\F, \G: \M\to\N$ are two functors in $\C\x\Mod$ that preserve colimits and $colim (\Ll_i)$ denotes certain colimit of
objects $\Ll_i\in\C\x\Mod$, we say that a natural transformation $\alpha: \F\to\G$ preserves the colimit if $\alpha(colim (\Ll_i))=colim(\alpha(\Ll_i))$.

To prove the mentioned equivalence we will use \cite[Proposition (1.1) in Chapter II]{Bass} by which we have:

\begin{prop} \prlabel{class in ab}
Let $\A$ be an abelian category with arbitrary coproducts and let $\Oo$ be a class of objects in $\A$ 
containing a generator of $\A$ and which is closed under cokernels and arbitrary coproducts of objects in $\Oo$. Then $\Oo$ contains all the objects of $\A$, i.e. $\Oo=Ob(\A)$.
\end{prop}

Given a  finite tensor category $\C$ observe that the category of left $\C$-module categories $\C\x\Mod$ has arbitrary coproducts. Namely, if $(\M_i\vert i\in J)$
is a family of $\C$-module categories, any object in the coproduct $\oplus_{i\in J}\M_i$ is a finite direct sum $M=\oplus_{i=1\\i\in J}^{i=n}M_i$ with $M_i\in\M_i$.
Then $\oplus_{i\in J}\M_i$
is a $\C$-module category by $C\crta\ot M:=\oplus_{i=1\\i\in J}^{i=n}(C\crta\ot M_i)$ for any $C\in\C$. Moreover, $\C$ is a generator of $\C\x\Mod$. Any $\C$-module category
is equivalent to $\C_A$ for an algebra $A\in\C$ by \cite[Theorem 3.17]{Os1}, \cite[Corollary 7.10.5]{EGNObook}. Given a non-zero functor $F:\C_A\to\N$ in $\C\x\Mod$, there is
an object $X\in\C_A$ such that $F(X)\not=0$, then define a functor $G:\C\to\C_A$ in $\C\x\Mod$ by $G(I)=X$. Then $FG:\C\to\N$ is a non-zero functor, which proves that $\C$ is a
generator of $\C\x\Mod$. Now we may proceed with:

\begin{thm} \thlabel{EW}
Let $\C$ and $\D$ be finite tensor categories. There is an equivalence of categories
$$\HH: \D\x\C\x\Bimod\to 2\x\Fun_{cont}(\C\x\Mod,\D\x\Mod).$$
\end{thm}

\begin{proof}
\dul{On 0-cells.}
Given a $\D\x\C$-bimodule category $\M$, there is a 2-functor $\HH(\M):=\M\Del_{\C}-: \C\x\Mod\to\D\x\Mod$. It sends
an object $\N\in\C\x\Mod$ to the object $\M\Del_{\C}\N\in\D\x\Mod$ and a functor $\F:\N_1\to\N_2$ \\ in $\C\x\Mod$ to the functor
$\M\Del_{\C}\F: \M\Del_{\C}\N_1\to\M\Del_{\C}\N_2\in\D\x\Mod$.
The functor $\M\Del_{\C}\F$ is well defined: it is $\C$-balanced, precisely because $\F$ is a left $\C$-module functor, and clearly $\M\Del_{\C}\F$ is left $\D$-linear.
The monoidal structure of $\M\Del_{\C}-$ for the composition of 1-cells is obviously the identity, so it is indeed a 2-functor.
Given a natural transformation $\alpha: \F\to\G$ in $\C\x\Mod$, we have a left $\D$-module natural transformation $\M\Del_{\C}\alpha: \M\Del_{\C}\F\to\M\Del_{\C}\G$.
The 2-functor $\HH(\M)=\M\Del_{\C}-$ is described in details in \cite[Theorem 3.4]{FMM}. Considered as a functor, it is left adjoint to the functor $\Fun_{\D}(\M, -)$,
thus it preserves colimits, in particular cokernels and arbitrary coproducts. Namely, by \cite[Corollary 3.22]{Gr} there is an equivalence of bimodule categories:
\begin{equation} \eqlabel{adj f}
\Fun_{\D}(\M\Del_{\C}\N, \Ll)\simeq\Fun_{\C}(\N, \Fun_{\D}(\M, \Ll)).
\end{equation}

\dul{On 1-cells.}
Given a $\D\x\C$-bimodule functor $\G:\M\to\N$ the corresponding 2-natural transformation in $2\x\Fun_{cont}(\C\x\Mod,\D\x\Mod)$ is given by $\HH(\G):=\G\Del_{\C}-: \M\Del_{\C}-\to\N\Del_{\C}-$,
which evaluated at any 
object $\Ll\in\C\x\Mod$ is given by the functor $\G\Del_{\C}\Ll$ in $\D\x\Mod$. The functor $\G\Del_{\C}\Ll$ is defined in the obvious way as the functor $\M\Del_{\C}\F$ above
(it is $\C$-balanced, since $\G$ is a right $\C$-module functor). We have that $\HH(\G)$ is a 2-natural transformation, for given
a functor $\F: \Ll_1\to\Ll_2$ in $\C\x\Mod$, the transformation: 
$\HH(\G)(\F): \HH(\G)(\Ll_2)\comp\HH(\M)(\F)\to\HH(\N)(\F)\comp\HH(\G)(\Ll_1)$ translates to: $(\G\Del_{\C}\Ll_2)(\M\Del_{\C}\F)\to(\N\Del_{\C}\F)(\G\Del_{\C}\Ll_1)$,
where both are obviously equal to the natural transformation $\G\Del_{\C}\F$, so the transformation $\HH(\G)(\F)$ can be taken to be identity (then it trivially
fulfilles the necessary coherences).

\bigskip

\dul{$\HH$ is essentially surjective.}
Take a 2-functor $\F$. Set $\Ll=\F(\N)$ in \equref{adj f} and consider $\F_{\C, \N}$ as an object on the right thereof (in \leref{left C-lin fun} we proved that it
is left $\C$-linear). Let $\G_{\C, \N}: \M\Del_{\C}\N\to\F(\N)$ be its companion on the left hand-side. Since the functors $\F_{\C, \N}$ are natural in $\N$, so are the
functors $\G_{\C, \N}$ and we may consider the natural transformation $\G_{\C, -}: \M\Del_{\C}-\to\F$. Let $\Oo$ be the class of $\C$-module categories $\N$ for which
$\G_{\C, \N}$ is an equivalence of left $\D$-module categories. Since both $\M\Del_{\C}-$ and $\F$ preserve cokernels and arbitrary coproducts, the class $\Oo$ is closed
for cokernels and arbitrary coproducts. Obviously $\C\x\Mod$ contains its generator, $\C$, then by \prref{class in ab} $\G_{\C, \N}$ is an equivalence for every $\C$-module category
$\N$, i.e. $\G_{\C, -}$ is a natural isomorphism. Observe that any natural transformation which operates over the corresponding underlying ordinary categories of 2-categories is a
2-natural transformation. Thus we have: $\F\simeq\M\Del_{\C}-$ as 2-functors.

\dul{$\HH$ is fully faithful.}
We clearly have that $\HH$ is faithful, i.e. the assignment
$$\HH_{\M, \N}: \Fun_{\D\x\C}(\M, \N)\to 2\x\Nat_{cont}(\M\Del_{\C}-, \N\Del_{\C}-)$$
$$\F\mapsto \F\Del_{\C}-$$
is injective, for given another $\G\in\Fun_{\D\x\C}(\M, \N)$ such that $\F\Del_{\C}-\simeq\G\Del_{\C}-$, it is $\F\simeq\F\Del_{\C}\C\simeq\G\Del_{\C}\C\simeq\G$.
Now, given $\alpha\in 2\x\Nat(\M\Del_{\C}-, \N\Del_{\C}-)$, we proved in \thref{pre-EW} that $\alpha(\C):\M\to\N$ is a $\D\x\C$-bimodule functor. To prove this we used the isomorphism \equref{s MC}
coming from the structure of a pseudo-natural transformation of $\alpha$ (which now is identity!) and the isomorphism $\theta$ in \equref{asoc a} coming from the structure of a
pseudo-functor of $\M\Del_{\C}-$ (which now is identity, too!). This means that we have that $\alpha(\C):\M\to\N$ is a $\D\x\C$-bimodule functor with the strict right $\C$-module functor structure,
where $\M$ and $\N$ are considered as strict right $\C$-module categories. Though, due to \cite[Proposition 2.8]{Ga1} this is not a problem, as we may consider $\D\x\C$-bimodule
categories $\M'$ and $\N'$ which are strict as right $\C$-module categories and such that $\M\simeq\M'$ and $\N\simeq\N'$. (If you are not comfortable with the non-strict left $\D$-module
structures, you may consider $\D\x\C$-bistructures as right $\C\Del\D^{rev}$-structures and replace $\M$ and $\N$ by equivalent strict right $\C\Del\D^{rev}$-module categories.)
Finally, similarly as in the proof of essential surjectiveness of $\HH$, we prove that $\HH_{\M, \N}(\alpha(\C))=\alpha(\C)\Del_{\C}-\simeq\alpha$ as natural transformations
$\M\Del_{\C}-\to \N\Del_{\C}-$. Let $\Oo$ be the class of $\Ll\in\C\x\Mod$ such that $\alpha(\C)\Del_{\C}\Ll\simeq\alpha(\Ll)$ as functors $\M\Del_{\C}\Ll\to \N\Del_{\C}\Ll$.
Then $\Oo$ contains the generator $\C$ and it is clearly closed for cokernels and arbitrary coproducts. Then by \prref{class in ab} we have that $\alpha(\C)\Del_{\C}\Ll\simeq\alpha(\Ll)$
for all $\Ll\in\C\x\Mod$, meaning that $\HH_{\M, \N}$ is a full functor.
\qed\end{proof}

\begin{rem}
Evaluating equation \equref{omega F_L} at $M\in\M$ yields an isomorphism in $\N\Del_{\C}\Ll$:
$$\tilde\omega(F_L)(M): \omega(\Ll)(M\Del_{\C}F_L)\stackrel{\iso}{\to}\omega(\C)(M)\Del_{\C}F_L.$$
Given that $F_L\in\Fun(\C,\Ll)_{\C}\simeq\Ll$ we may rewrite this as:
$$\tilde\omega(L)(M): \omega(\Ll)(M\Del_{\C}L)\stackrel{\iso}{\to}\alpha(M)\Del_{\C}L$$
where $\alpha=\omega(\C)$, or more generally:
\begin{equation} \eqlabel{omega i alfa}
\tilde\omega(\bullet)(-): \omega(\Ll)(-\Del_{\C}\bullet)\stackrel{\iso}{\to}\alpha(-)\Del_{\C}\Id_{\Ll}(\bullet).
\end{equation}
\end{rem}

\begin{rem}
The right hand-side version of the above Theorem is proved similarly using the adjunction $\Fun(\N\Del_{\C}\M, \A)_{\E}\iso\Fun(\N, \Fun(\M, \A)_{\E})_{\C}$,
proved in \cite[Proposition 4.5]{Femic2}.
\end{rem}

On one hand, from \thref{EW} it is clear that there is a bijection between 
equivalence functors $\C\x\Mod\to\D\x\Mod$ and invertible $\D\x\C$-bimodule categories.
On the other hand, if $\C=\D$ both categories in \thref{EW} are monoidal: $(\C\x\Bimod, \Del_{\C}, \C)$ and $\Fun_{cont}(\C\x\Mod,\C\x\Mod)$ is monoidal with respect to
the composition of functors and unit being the identity functor on $\C\x\Mod$. 
We have a monoidal equivalence of categories $\C\x\Bimod$ and $\Fun_{cont}(\C\x\Mod,\C\x\Mod)$,
because of the canonical equivalence of bimodule categories. If moreover $\C$ is braided, the monoidal subcategory of one-sided $\C$-bimodule categories $(\C^{br}\x\Mod, \Del_{\C}, \C)$
turns out to be monoidally equivalent to the subcategory $\Fun_{cont}^*(\C\x\Mod,\C\x\Mod)$ of $\Fun_{cont}(\C\x\Mod,\C\x\Mod)$ of ``continuos'' endofunctors of $\C\x\Mod$
with respect to the one-sided $\C$-module structures.

Extracting invertible objects and morphisms in the above monoidal categories we get the categorical groups:
$\ul{\BrPic}(\C)$, \cite[Section 4]{ENO}, and $\ul{\AUT}(\C\x\Mod)$ - the categorical group of 
equivalence endofunctors of $\C\x\Mod$. Truncating them to ordinary
groups (taking corresponding isomorphism classes of objects) we obtain the Brauer-Picard group $\BrPic(\C)$ and $\AUT(\C\x\Mod)$ the group of isomorphism classes of
equivalence endofunctors of $\C\x\Mod$, where the isomorphism classes are taken with respect to natural isomorphisms between functors. These groups are isomorphic,
as well as their corresponding subgroups: the Picard group $\Pic(\C)$ and $\AUT^*(\C\x\Mod)$ (with one-sided $\C$-structures).

\begin{cor} \colabel{autofun}
For a finite tensor category $\C$ there is a group isomorphism
$$\BrPic(\C)\to\AUT(\C\x\Mod)\quad\textnormal{given by}\quad [\M]\mapsto[\M\Del_{\C}-]$$
which restricts to the isomorphism of subgroups $\Pic(\C)\iso\AUT^*(\C\x\Mod)$.
\end{cor}

\section{Bialgebroid categories} \selabel{Bialgebroids}

Let $\C$ be a finite tensor category. Given objects $X_1\dots X_n$, we will often write $X_1\cdots X_n$ for the tensor product
$X_1\ot \dots\ot X_n$ in $\C$. The n-fold Deligne tensor product $\C^{\Del n}$ is a tensor category with the componentwise tensor product and
unit object $I^{\Del n}$, where $I$ is the unit object of $\C$.

In  \cite[Section 4]{Femic3} we defined a $\C$-comonoidal category. It is a 5-tuple $(\A, \Delta, \varepsilon, a, l, r)$ where $\A$
is a $\C$-bimodule category, $\Delta: \A\to\A\Del_{\C}\A$ and $\varepsilon: \A\to\C$ are $\C$-bimodule functors, and $a,l,r$ are certain $\C$-bimodule
natural isomorphisms, so that certain 2 coherence diagrams commute. We use the Sweedler-type notation $\Delta(A)=A_{(1)}\Del_{\C} A_{(2)}$ for $A\in\A$.
Intuitively, a $\C$-bialgebroid category would be a monoidal and a $\C$-comonoidal category such that the comultiplication and the counit functors of $\A$ 
are monoidal functors. However, one should be cautious about how to define a tensor product in $\A\Del_{\C}\A$. We present two definitions for a $\C$-bialgebroid, 
one for a non-braided and the other for a braided category $\A$. For $\A$ braided these two definitions basically coincide: the first one is slightly more general - 
we do not specify in it how the $\C$-bimodule category structure of $\A$ is given. 

\medskip

In \cite[Section 3.3]{Neu} bimonoidal categories were defined. Let $(\Cc_k, \Del, vec)$ denote the 2-category of (small) finite abelian categories, where $vec$
is the category of finite-dimensional $k$-vector spaces. A bimonoidal category is a monoidal category with a comonoidal structure in $(\Cc_k, \Del, vec)$ with certain
compatibility conditions for the monoidal and the comonoidal structure. A $\C$-bialgebroid category is a monoidal category with a comonoidal structure in the 2-category
$(\C\x\Bimod, \Del_{\C}, \C)$, rather than in
$(\Cc_k, \Del, vec)$, with similar compatibility conditions. Concretely, compared to \cite[Definition 3.8]{Neu}, we define a $\C$-bialgebroid category requiring a
different natural transformation $\Sigma$ and basically replacing the category $vec$ by $\C$ and the Deligne tensor product $\Del$ by $\Del_{\C}$
throughout. To save space, we will not type the necessary 12 coherence diagrams: they will be the analogous 12 coherence diagrams as in \cite[Definition 3.8]{Neu}
where the described changes are made.

In the next definition we will use the following functor:
$$\Delta *\Delta: \A\Del\A\to\A\Del_{\C}\A \qquad\textnormal{given by}\quad\Delta *\Delta(A\Del B)=A_{(1)}B_{(1)}\Del_{\C} A_{(2)}B_{(2)}.$$

\medskip

\begin{defn} \delabel{bialg noncomm}
Let $\C$ and $\A$ be finite tensor categories where $(\A, \ot, I, \alpha, \tilde\lambda, \tilde\rho)$ is a monoidal category structure of $\A$. We say that $\A$ is a 
{\em $\C$-bialgebroid category} if $(\A, \Delta, \varepsilon, a, l, r)$ is a $\C$-comonoidal category so that there are natural isomorphisms:
\begin{enumerate}
\item $\Sigma: \Delta *\Delta \to \Delta\comp\ot$, 
that is, for all $A, B\in\A$ there is an isomorphism: \\ $A_{(1)}B_{(1)}\Del_{\C} A_{(2)}B_{(2)}\stackrel{\Sigma}{\to} (AB)_{(1)}\Del_{\C}(AB)_{(2)}$;
\item $\Theta: \ot_{\C}\comp(\varepsilon\Del\varepsilon) \to \varepsilon\comp\ot_{\A}$;
that is, for all $A, B\in\A$ there is an isomorphism: $\varepsilon(A)\varepsilon(B)\stackrel{\Theta}{\to} \varepsilon(AB)$;
\end{enumerate}
and isomorphisms:
\begin{enumerate}[(i)]
\item $\Sigma_0: I\Del_{\C}I\stackrel{\Sigma_0}{\to}\Delta(I)=I_{(1)}\Del_{\C}I_{(2)}$;
\item $\Theta_0: I_{\C}\stackrel{\Theta_0}{\to}\varepsilon(I)$,
where $I_{\C}$ is the unit object in $\C$;
\end{enumerate}
so that the 12 coherence diagrams described above commute.
\end{defn}

We comment briefly the contents of the 12 compatibility diagrams:
\begin{enumerate}
\item $\Sigma^2, \Delta^2$;
\item $\Delta(\tilde\lambda)$ (with $\Sigma, \Sigma_0, \Delta$);
\item $\Delta(\tilde\rho)$ (with $\Sigma, \Sigma_0, \Delta$);
\item $\Theta^2$ (with $\varepsilon$);
\item $\varepsilon(\tilde\lambda)$ (with $\Theta_0, \Theta$);
\item $\varepsilon(\tilde\rho)$ (with $\Theta_0, \Theta$);
\item $a, \Delta, \Sigma$;
\item $l\ot l$ (with $\Theta, \varepsilon, \Sigma$);
\item $r\ot r$ (with $\Theta, \varepsilon, \Sigma$);
\item $\Sigma_0^2, \Delta, a$;
\item $\varepsilon, \Sigma_0, \Theta_0, l$;
\item $\varepsilon, \Sigma_0, \Theta_0, r$.
\end{enumerate}

\medskip

In order to make a faithful generalization of the notion of a bialgebroid, actually we would have had to require $\A$ to be a bimodule category
over $\C\Del\C^{rev}$. However, this kind of construction goes beyond our interest at this moment.
We will be interested in a $\C$-bialgebroid category $\A$ which is symmetric as a monoidal category.
In \cite[Lemma 2.8]{Femic3} we proved for a braided tensor category $\A$ which has a structure of a $\C$-bimodule category coming from two
tensor functors $\lambda, \rho:\C\to\A$, that $\A\Del_{\C}\A$ is a braided finite tensor category with the tensor product
$$
M: (\A\Del_{\C}\A)\Del(\A\Del_{\C}\A)\to\A\Del_{\C}\A
$$
induced by a functor $\C$-balanced at two places which is given by  $(A\Del B)\Del(A'\Del B'):=AA'\Del_{\C} BB'$, for $A\Del B, A'\Del B'\in\A\Del\A$. 
When proving the $\C$-balance of this
operation, we kept track of the directions in which we moved objects from $\C$, so we did not encounter symmetricity restrictions on the braiding,
as it can be seen in the proof. To simplify this we will assume, as it is done in \cite{DGNO1}
and \cite{Gr1}, that the functors $\lambda$ and $\rho$ factor through the M\"uger's center category $\A'$. The M\"uger's center category $\A'$ is
a braided subcategory of $\A$ such that for all $A'\in\A'$ and all $A\in\A$ the braiding $\tilde\Phi$ of $\A$ is symmetric when acting between
$A'$ and $A$, that is: $\tilde\Phi_{A,A'}=(\tilde\Phi_{A,A'})^{-1}$, \cite[Definition 2.9]{M}.

\begin{defn} \delabel{second def}
Let $\C$ and $\A$ be finite braided tensor categories where $(\A, \ot, I, \alpha, \tilde\lambda, \tilde\rho)$ is a monoidal category structure of $\A$. We say that $\A$ is a 
{\em braided $\C$-bialgebroid category} if:
\begin{enumerate}
\item there are braided tensor functors $\lambda, \rho: \C\to\A$ which factor through the M\"uger's center category $\A'$ and
provide $\A$ with a $\C$-bimodule category structure:
$$X\crta\ot A=\lambda(X)A, \qquad A\crta\ot X=A\rho(X)$$
where $X\in\C, A\in\A$; 
\item $(\A, \Delta, \varepsilon, a, l, r)$ is a $\C$-comonoidal category with the above $\C$-bimodule category structure;
\item there are natural isomorphisms:
\begin{enumerate}
\item $\Sigma: M\comp(\Delta\Del\Delta) \to \Delta\comp\ot$,
that is, for all $A, B\in\A$ there is an isomorphism: \\ $A_{(1)}B_{(1)}\Del_{\C} A_{(2)}B_{(2)}\stackrel{\Sigma}{\to} (AB)_{(1)}\Del_{\C}(AB)_{(2)}$;
\item $\Theta: \ot_{\C}\comp(\varepsilon\Del\varepsilon) \to \varepsilon\comp\ot_{\A}$;
that is, for all $A, B\in\A$ there is an isomorphism: \\ $\varepsilon(A)\varepsilon(B)\stackrel{\Theta}{\to} \varepsilon(AB)$;
\end{enumerate}
and isomorphisms:
\begin{enumerate}[(i)]
\item $\Sigma_0: I\Del_{\C}I\stackrel{\Sigma_0}{\to}\Delta(I)=I_{(1)}\Del_{\C}I_{(2)}$,
where $I$ is the unit object in $\A$;
\item $\Theta_0: I_{\C}\stackrel{\Theta_0}{\to}\varepsilon(I)$,
where $I_{\C}$ is the unit object in $\C$;
\end{enumerate}
so that the 12 coherence diagrams as in \deref{bialg noncomm} commute.
\end{enumerate}
\end{defn}

\begin{rem}
The only differences between the above two definitions are that in \deref{second def} we specify the $\C$-bimodule category structures and that
the natural isomorphism $\Sigma$ is formally differently defined.
\end{rem}

If $\A$ is a $\C$-bimodule category one may define tensor functors $\lambda, \rho: \C\to\A$ by $\lambda(C)=C\crta\ot I$ and $\rho(C)=I\crta\ot C$. 
Then the $\C$-bimodule category structure on $\A$ defined as in item 1) of the above definition coincides with the original one. 

\begin{lma} \lelabel{res sc}
Every $\A$-bimodule category $\M$ is a $\C$-bimodule category ${}_{\lambda}\M_{\rho}$ by the restriction of scalars functor:
$$X\crta\ot M=\lambda(X)\crta\ot M, \qquad M\crta\ot X=M\crta\ot\rho(X)$$
for every $X\in\C, M\in\M$.
\end{lma}

\begin{lma}
For any finite symmetric tensor category $\C$ the category $\A=\C\Del\C$ is a $\C$-bialgebroid category with the functors
$\Delta: \C\Del\C \to (\C\Del\C) \Del_{\C} (\C\Del\C) \simeq \C\Del\C\Del\C$ and $\varepsilon:\C\Del\C\to\C$ given by
$$\Delta(X\Del Y)=(X\Del I) \Del_{\C} (I\Del Y)\quad\textnormal{and}\quad \varepsilon(X\Del Y)=XY.$$
\end{lma}

\begin{proof}
In \cite[Lemma 4.3]{Femic3} we proved that $\C\Del\C$ is a $\C$-comonoidal category with the given structures and the obvious $\C$-bimodule category structure. Set
$$\lambda(X)=X\Del I\quad\textnormal{and}\quad\rho(X)=I\Del X$$
for $X\in\C$ where $I$ represents the unit object in $\C$. To check the condition 3) in \deref{second def}, observe that
$\Sigma, \Sigma_o$ and $\Theta_0$ can be taken to be identities. Observe that the functors in the parts (a) and (b) are defined on objects of the form $A=X\Del Y, B=X'\Del Y'\in\A$. 
Note that $\ot_{\C}\comp(\varepsilon\Del\varepsilon)((X\Del Y)\odot(X'\Del Y'))
=XYX'Y'$ and $\varepsilon\comp\ot_{\A}((X\Del Y)\odot(X'\Del Y'))=XX'YY'$, where $\odot$ denotes the tensor product in $\A$. Define
$$\Theta=X\ot\Phi_{Y, X'}\ot Y': XYX'Y'=\varepsilon(X\Del Y)\varepsilon(X'\Del Y')\to\varepsilon(XX'\Del YY')=XX'YY'.$$
It is sufficient to check the coherence diagrams 4,5,6,8 and 9. Though, the diagrams 5, 6, 8 and 9 commute clearly, so let us see the diagram 4. It comes down to:
$$
\gbeg{9}{7}
\got{1}{\varepsilon(X\Del Y)} \gvac{2} \got{1}{\varepsilon(X'\Del Y')} \gvac{2} \got{2}{\varepsilon(X''\Del Y'')}\gnl
\gcl{1} \gvac{2} \gcl{1} \gvac{2} \gcl{1} \gnl
\gcl{1} \gvac{2} \glmpt\gcmp\gcmpb\gnot{\hspace{-1,2cm}\Theta}\grmpt \gnl
\gcn{1}{1}{1}{3} \gvac{3} \gcn{1}{1}{3}{1} \gnl
\gvac{1} \glmpt\gcmp\gcmpb\gnot{\hspace{-1,2cm}\Theta}\grmpt \gnl
\gvac{3} \gcl{1} \gnl
\gvac{3} \gob{1}{\varepsilon(XX'X''\Del YY'Y'')}
\gend=
\gbeg{9}{7}
\gvac{1} \got{1}{\varepsilon(X\Del Y)} \gvac{2} \got{1}{\varepsilon(X'\Del Y')} \gvac{2} \got{2}{\varepsilon(X''\Del Y'')}\gnl
\gvac{1} \gcl{1} \gvac{2} \gcl{1} \gvac{2} \gcl{1} \gnl
\gvac{1} \glmpt\gcmpb\gcmp\gnot{\hspace{-1,2cm}\Theta}\grmpt \gvac{2} \gcl{1} \gnl
\gvac{2} \gcn{1}{1}{1}{3} \gvac{3} \gcn{1}{1}{3}{1} \gnl
\gvac{3} \glmpt\gcmpb\gcmp\gnot{\hspace{-1,2cm}\Theta}\grmpt \gnl
\gvac{4} \gcl{1} \gnl
\gvac{4} \gob{1}{\varepsilon(XX'X''\Del YY'Y'').}
\gend
$$
Let us denote symbolically:
$$\Theta=
\gbeg{5}{4}
\got{1}{X} \got{1}{Y} \got{1}{X'} \got{1}{Y'} \gnl
\gcl{1} \gbr \gcl{1} \gnl
\gmu \gmu \gnl
\gob{2}{XX'} \gob{2}{YY'}
\gend
$$
then the above equation becomes:
$$
\gbeg{6}{7}
\got{1}{X} \got{1}{Y} \got{1}{X'} \got{1}{Y'} \got{1}{X''} \got{1}{Y''} \gnl
\gcl{4} \gcl{3} \gcl{1} \gbr \gcl{1} \gnl
\gvac{2} \gmu \gmu \gnl
\gvac{2} \gcn{1}{1}{2}{1} \gvac{1} \gcn{1}{1}{2}{1}  \gnl
\gvac{1} \gbr \gvac{1} \gcl{1} \gnl
\gmu \gwmu{3} \gnl
\gob{2}{X(X'X'')} \gvac{1} \gob{2}{Y(Y'Y'')}
\gend=
\gbeg{6}{7}
\got{1}{X} \got{1}{Y} \got{1}{X'} \got{1}{Y'} \got{1}{X''} \got{1}{Y''} \gnl
\gcl{1} \gbr \gcl{1} \gcl{3} \gcl{4} \gnl
\gmu \gmu \gnl
\gcn{1}{1}{2}{3} \gvac{1} \gcn{1}{1}{2}{3}   \gnl
\gvac{1} \gcl{1} \gvac{1} \gbr \gnl
\gvac{1} \gwmu{3} \gmu \gnl
\gvac{1} \gob{2}{(XX')X''} \gvac{2} \gob{1}{(YY')Y''}
\gend
$$
which is obviously fulfilled by naturality and the associativity constraint of $\C$.
\qed\end{proof}

When $\C$ and a $\C$-bialgebroid category $\A$ are symmetric as monoidal categories (as for example in the above Lemma), we will say that $\A$ is a {\em symmetric $\C$-bialgebroid category}.

\section{Harrison cohomology for a symmetric bialgebroid category} \selabel{Harrison}

Let $\A=(\A,\C,\lambda,\rho;\ot, I;\Delta,\varepsilon)$ be a $\C$-bialgebroid category. For $n\in \Nn$ we write
$\underbrace{\A\Del_{\C}\cdots \Del_{\C}\A}_n=\A^{\Del_{\C}n}$ and we set $\A^{\Del_{\C} 0}=\C$. Let
\begin{equation} \eqlabel{e_i's}
e_i^n:\ \A^{\Del_{\C} n}\to \A^{\Del_{\C}(n+1)}
\end{equation}
for $i\in \{0,\cdots,n+1\}$ be functors given by:
$$e_0^n=\rho\Del_{\C} \A^{\Del_{\C}n}~~;~~e_{n+1}^n=\A^{\Del_{\C} n}\Del_{\C} \lambda;$$
$$e_i^n=\A^{\Del_{\C}(i-1)}\Del_{\C}\Delta\Del_{\C}\A^{\Del_{\C}(n-i)},$$
for $i=1,\cdots, n$, that is:
\begin{eqnarray*}
e_0^n(A^1\Del_{\C}\cdots\Del_{\C} A^n)&=& I\Del_{\C} A^1\Del_{\C}\cdots\Del_{\C} A^n;\\
e_i^n(A^1\Del_{\C}\cdots\Del_{\C} A^n)&=&A^1\Del_{\C}\cdots\Del_{\C}\Delta(A^i)\Del_{\C}\cdots\Del_{\C} A^n;\\
e_{n+1}^n(A^1\Del_{\C}\cdots\Del_{\C} A^n)&=&A^1\Del_{\C}\cdots\Del_{\C} A^n\Del_{\C}I
\end{eqnarray*}
for every object $A^1\Del_{\C}\cdots\Del_{\C} A^n\in\A^{\Del_{\C}n}$. (Here we are abusing the notation: the functors $e^n_i$ are actually induced by the 
corresponding functors $\A^{\Del n}\to \A^{\Del_{\C}(n+1)}$ which are $\C$-balanced at $n-1$ places.) The definition on morphisms is similar.

For a consequence of the coassociativity of the functor $\Delta$ we have:

\begin{lma} \lelabel{e_i's rel}
For $i\geq j\in \{0,1,\cdots,n+1\}$ there is a natural isomorphism of functors:
\begin{equation}\eqlabel{2.1.1.1}
e_j^{n+1}\comp e_i^n \simeq e_{i+1}^{n+1}\circ e_j^n.
\end{equation}
\end{lma}

Assume from now on that 
$(\A,\C,\lambda,\rho;\ot, I;\Delta,\varepsilon)$ is a symmetric $\C$-bialgebroid category. 
By \cite[Lemma 2.8]{Femic2} and as we recalled in the last section, the category $\A\Del_{\C}\A$ is then symmetric tensor and similarly so are the categories
$\A^{\Del_{\C}n}$ for all $n\in\Nn$. The functors \equref{e_i's} are now symmetric tensor functors. 
Let $P$ be an additive covariant functor from symmetric tensor categories to abelian groups. Then we consider
$$\delta_n=\sum_{i=0}^{n+1} (-1)^{i-1}P(e^n_i):\ P(\A^{\Del_{\C} n})\to P(\A^{\Del_{\C} (n+1)}).$$
By \leref{e_i's rel} one shows that $\delta_{n+1}\circ \delta_n=0$, so we obtain a complex:
$$ \bfig \putmorphism(0, 0)(1, 0)[0`P(\C)`]{420}1a
\putmorphism(400, 0)(1, 0)[\phantom{P(\C)}`P(\A)`\delta_0]{500}1a
\putmorphism(900, 0)(1, 0)[\phantom{P(\C)}`P(\A^{\Del_{\C} 2})`\delta_1]{560}1a
\putmorphism(1460, 0)(1, 0)[\phantom{P(\A^{\Del_{\C} 2})}` P(\A^{\Del_{\C} 3})`\delta_2]{660}1a
\putmorphism(2100, 0)(1, 0)[\phantom{P(\A^{\Del_{\C}3})}`\cdots`\delta_3]{550}1a
\efig
$$
We will call it {\em Harrison complex $C(\A/\C, P)$}. We have:
$$Z^n(\A, P)=\Ker\delta_n,~~B^n(\A, P)=\im\delta_{n-1}~~~\textnormal{and}~~~H^n(\A, P)=Z^n(\A, P)/B^n(\A, P).$$
We will call $H^n(\A, P)$ the $n$-th Harrison cohomology group of $\A$ with values in $P$. Elements in $Z^n(\A,P)$ are called $n$-cocycles, and elements
in $B^n(\A,P)$ are called $n$-coboundaries.

\begin{rem} \rmlabel{Amitsur-Harrison}
For a symmetric $\C$-bialgebroid category $\A=\C\Del\C$ an object $A^1\Del_{\C}\cdots\Del_{\C} A^n\in\A^{\Del_{\C}n}$ corresponds to an object
$X^1\Del\cdots\Del X^{n+1}\in\C^{\Del (n+1)}$ and the above tensor functors $e_i^n:\ \A^{\Del_{\C} n}\to \A^{\Del_{\C}(n+1)}$ for $i=0,1,\cdots,n+1$ coincide
with the tensor functors $e_j^{n+1}:\ \C^{\Del (n+1)}\to \C^{\Del(n+2)}$ for $j=1,\cdots,n+2$ from \cite{Femic2}, so that the Harrison cohomology for $\A$ reduces to
the Amitsur cohomology for $\C$ and we have: $H^n(\A, P)=H^{n+1}(\C, P)$.
\end{rem}

Observe that for $\C$ and $\A$ symmetric the category $\dul{\Pic}(\A^{\Del_{\C} n})$ is symmetric monoidal by \cite[Proposition 3.4]{Femic2}.
The functors \equref{e_i's} induce the functors
$$E^n_i:\ \dul{\Pic}(\A^{\Del_{\C} n})\to \dul{\Pic}(\A^{\Del_{\C} (n+1)})$$
which are formally defined in the same fashion as in \cite[Section 5.1]{Femic2}, that is:
\begin{equation}\eqlabel{M_i}
E^n_i(\M)=\M_i=\M\Del_{\A^{\Del_{\C} n}} {}_{e^n_i}\A^{\Del_{\C} (n+1)}
\end{equation}
and
$$E^n_i(F)=F_i=F\Del_{\A^{\Del_{\C} n}} {}_{e^n_i}\A^{\Del_{\C} (n+1)}$$
for every object $\M$, every functor $F$ in $\dul\Pic(\A^{\Del_{\C} n})$ and $i=0, \cdots, n+1$. Here we use the notation from \equref{res}.

The construction of the Harrison cohomology groups $H^n(\A,\dul{\Pic})$
is done {\em mutatis mutandis} as that of the Amitsur cohomology groups $H^n(\C,\dul{\Pic})$. Namely, we have, as in \cite[Lemma 5.6]{Femic2} that for
$i\geq j\in \{1,\dots,n+1\}$ and $\M\in \dul{\Pic}(\A^{\Del_{\C} n})$ there is a natural equivalence:
\begin{equation}\eqlabel{2.2.1.1b}
\M_{ij}\simeq \M_{j(i+1)}
\end{equation}
where $\M_{ij}=E_j^{n+1}\comp E_i^n(\M)$. For every non-zero $n\in\mathbb N$, we define a functor
\begin{equation} \eqlabel{delta_n Pic}
\delta_{n}:\ \dul{\Pic}(\A^{\Del_{\C} n})\to \dul{\Pic}(\A^{\Del_{\C} (n+1)}),
\end{equation}
by
$$\delta_{n}(\M)=\M_0\Del_{\A^{\Del_{\C} (n+1)}}\M^{op}_1\Del_{\A^{\Del_{\C} (n+1)}}\cdots
\Del_{\A^{\Del_{\C} (n+1)}}\N_{n+1},$$
$$\delta_{n}(F)=F_0\Del_{\A^{\Del_{\C} (n+1)}} (F^{op}_1)^{-1}\Del_{\A^{\Del_{\C} (n+1)}}\cdots \Del_{\A^{\Del_{\C} (n+1)}} (G_{n+1})^{\pm 1},$$
with $\N=\M$ or $\M^{op}$ and $G=F$ or $F^{op}$ depending on whether $n$ is even or odd.
We also have:
\begin{equation}\eqlabel{double delta M}
\delta_{n+1}\delta_{n}(\M)=(\boxtimes_{\A^{\Del_{\C} (n+2)}})_{j=1}^{n+2}(\boxtimes_{\A^{\Del_{\C} (n+2)}})_{i=0}^{j-1} \hspace{0,2cm} (\M_{ij}\Del_{\A^{\Del_{\C} (n+2)}} \M_{ij}^{op}),
\end{equation}
\begin{equation}\eqlabel{double delta F}
\delta_{n+1}\delta_{n}(F)=(\boxtimes_{\A^{\Del_{\C} (n+2)}})_{j=1}^{n+2}(\boxtimes_{\A^{\Del_{\C} (n+2)}})_{i=0}^{j-1} \hspace{0,2cm} (F_{ij}\Del_{\A^{\Del_{\C} (n+2)}} (F_{ij}^{op})^{-1}),
\end{equation}
so we have a natural equivalence:
$$\lambda_{\M}= (\boxtimes_{\A^{\Del_{\C} (n+2)}})_{j=1}^{n+2}(\boxtimes_{\A^{\Del_{\C} (n+2)}})_{i=0}^{j-1} \ev_{\M_{ij}}:\ \delta_{n+1}\delta_n(\M)\to \A^{\Del_{\C} (n+2)}.$$

Now we define the category $\dul{Z}^n(\A,\dul{\Pic})$ whose objects are pairs $(\M,\alpha)$, where $\M\in \dul{\Pic}(\A^{\Del_{\C} n})$ and
$\alpha:\ \delta_n(\M)\to \A^{\Del_{\C}(n+1)}$ is an equivalence of $\A^{\Del_{\C}(n+1)}$-module categories such that $\delta_{n+1}(\alpha)\simeq\lambda_{\M}$.
A morphism $(\M,\alpha)\to (\N,\beta)$ is given by an equivalence $F:\M\to \N$ of $\A^{\Del_{\C} n}$-module categories satisfying $\beta\circ \delta_n(F)\simeq\alpha$.
One has that $\dul{Z}^n(\A,\dul{\Pic})$ is a symmetric monoidal category, whose tensor product is given by: $(\M,\alpha)\ot (\N,\beta)=
(\M\Del_{\A^{\Del_{\C} n}}\N,\alpha\Del_{\A^{\Del_{\C} (n+1)}}\beta)$ and unit object by $(\A^{\Del_{\C} n}, \A^{\Del_{\C}(n+1)})$. Moreover, all the objects in
$\dul{Z}^n(\A,\dul{\Pic})$ are invertible. Its Grothendieck group:
$$K_0\dul{Z}^n(\A,\dul{\Pic})={Z}^n(\A,\dul{\Pic})$$
is the group of n-cocycles. Let
$$d_{n-1}:\ \dul{\Pic}(\A^{\Del_{\C} (n-1)})\to \dul{Z}^n(\A,\dul{\Pic})$$
be a strongly monoidal functor given by:
$d_{n-1}(\N)=(\delta_{n-1}(\N),\lambda_{\N})$. Denote by $B^n(\A,\dul{\Pic})$ the subgroup of $Z^n(\A, \dul{\Pic})$, whose elements represented by $d_{n-1}(\N)$,
for some $\N\in\dul{\Pic}(\A^{\Del_{\C} (n-1)})$. We finally define:
$$H^n(\A,\dul{\Pic})=Z^n(\A,\dul{\Pic})/B^n(\A,\dul{\Pic}).$$

Similarly to what we commented in \rmref{Amitsur-Harrison}, for $\A=\C\Del\C$ we have: $H^n(\A, \dul{\Pic})=H^{n+1}(\C, \dul{\Pic})$. In the next section
we will interpret the cohomologu group $H^2(\A, \dul{\Pic})$, so that in the case  $\A=\C\Del\C$ we get an interpretation for $H^3(\C, \dul{\Pic})$.

\medskip

The analogous result to \cite[Theorem 6.2]{Femic2} holds for a symmetric $\C$-bialgebroid category $\A$. The Sections 5 and 6 of \cite{Femic2}
can be done {\em mutatis mutandis} for $\A$. This bialgebroid version is a categorification of \cite[Theorem 9.3.5]{th}.

\begin{thm}\thlabel{VZb}
Let $\C$ be a symmetric finite tensor category and $\A$ a symmetric $\C$-bialgebroid category. There is a long exact sequence
\begin{eqnarray}\eqlabel{VZ seq}
1&\longrightarrow&H^1(\A,\Inv)\stackrel{\alpha_1}{\longrightarrow}H^0(\A,\dul{\Pic})\stackrel{\beta_1}{\longrightarrow}H^0(\A,{\Pic})\\
&\stackrel{\gamma_1}{\longrightarrow}&H^2(\A,\Inv)\stackrel{\alpha_2}{\longrightarrow}H^1(\A,\dul{\Pic})\stackrel{\beta_2}{\longrightarrow}H^1(\A,{\Pic})\nonumber\\
&\stackrel{\gamma_2}{\longrightarrow}&H^3(\A,\Inv)\stackrel{\alpha_3}{\longrightarrow}H^2(\A,\dul{\Pic})\stackrel{\beta_3}{\longrightarrow}H^2(\A,{\Pic})\nonumber\\
&\stackrel{\gamma_3}{\longrightarrow}&\cdots \nonumber
\end{eqnarray}
\end{thm}

The functors $P$ from symmetric tensor categories to abelian groups appearing in the above theorem are $\Inv$ and $\Pic$ in the first and the third column, respectively.
We saw that $\Pic(\A)$ is the Picard group of $\A$, and $\Inv(\A)$ is the group of isomorphism classes of invertible objects in $\A$.

\section{Quasi-monoidal structures on $\A\x\Mod$} \selabel{Quasi-mon}

Let $\A$ be a symmetric $\C$-bialgebroid category. By a quasi-monoidal structure on $\A\x\Mod$ we will understand a tensor product on the 0-cells of the 2-category $\A\x\Mod$ 
together with an associativity constraint satisfying the pentagonal axiom up to a natural equivalence.
In this section we will present a way to generate quasi-monoidal structures on $\A\x\Mod$.
Here ``quasi'' on the one hand
alludes to the absence of the unit 0-cell and the corresponding unity constraints, and on the other hand, it refers to the fact that we are not interested in the rest of the axioms 
for a monoidal 2-category structure on $\A\x\Mod$.
From the point of view of a monoidal structure on the 2-category $\A\x\Mod$, defined in \cite{KV}, our quasi-monoidal structure means a tensor
product (M2) satisfying the axiom (M9), as they are numerated in the definition presented in \cite{Gr}.

We will consider every $\A$-module category as a one-side $\A$-bimodule category, because $\A$ is symmetric. On the other hand, by \leref{res sc}
we will consider every $\A$-module category as a $\C$-bimodule category (not one-sided!). Moreover, in this setting $\A\Del_{\C}\A$ is also symmetric and
we will consider every left $\A\Del_{\C}\A$-module category as one-sided $\A\Del_{\C}\A$-bimodule category.

\begin{lma}
There is a functor $-\Del_{\C}-: \A\x\Mod\times\A\x\Mod\to\A\Del_{\C}\A\x\Mod$ defined on objects and morphisms as follows:
$$(\M, \N)\mapsto\M\Del_{\C}\N\quad\textnormal{and}\quad(\F, \G)\mapsto\F\Del_{\C}\G.$$
\end{lma}

\smallskip

In our setting the comultiplication functor $\Delta:\A\to\A\Del_{\C}\A$ is a tensor functor, let us consider the restriction of scalars functor $\R: \A\Del_{\C}\A\x\Mod\to\A\x\Mod$.
We have that any left $\A\Del_{\C}\A$-module category $\M$ is a left $\A$-module category via $\Delta$, a fact which we will denote by ${}_{\Delta}\M$, as in \equref{res}.

Let us now observe the following composition of functors:
\begin{equation} \eqlabel{tensor functor}
\bfig
\putmorphism(-400, 0)(1, 0)[\A\x\Mod\times\A\x\Mod` \A\Del_{\C}\A\x\Mod` -\Del_{\C}-]{1250}1a
\putmorphism(850, 0)(1, 0)[\phantom{\A\Del_{\C}\A\x\Mod} ` \A\Del_{\C}\A\x\Mod ` F]{950}1a
\putmorphism(1800, 0)(1, 0)[\phantom{\A\Del_{\C}\A\x\Mod}` \A\x\Mod ` \R]{800}1a
\efig
\end{equation}
where $F$ is an autoequivalence functor of $\A\Del_{\C}\A\x\Mod$. Recall that it is given by $F\iso\Pp\Del_{\A\Del_{\C}\A}-$ for some invertible
one-sided $\A\Del_{\C}\A$-bimodule category $\Pp$, by \coref{autofun}, (observe that $\Pp=F(\A\Del_{\C}\A)$).
For $\M, \N\in\A\x\Mod$ the resulting $\A$-module category from the above composition we will denote by:
\begin{equation} \eqlabel{new tensor}
\M\boxdot_F\N={}_{\Delta}(\Pp\Del_{\A\Del_{\C}\A}(\M\Del_{\C}\N))={}_{\Delta}\Pp\Del_{\A\Del_{\C}\A}(\M\Del_{\C}\N).
\end{equation}
Observe that by the definition of $\M\boxdot_F\N$ its arbitrary object is of the form: $P\Del_{\A\Del_{\C}\A}(M\Del_{\C}N)$ for $M\in\M, N\in\N$ and $P\in\Pp$. 
Moreover, given $A\in\A$, we have:
\begin{multline*}
A\crta\ot (P\Del_{\A\Del_{\C}\A}(M\Del_{\C}N))=((A_{(1)}\Del_{\C}A_{(2)}) \crta\ot P) \Del_{\A\Del_{\C}\A}(M\Del_{\C}N)\iso \\
(P\Del_{\A\Del_{\C}\A}(A_{(1)}\Del_{\C}A_{(2)}))\Del_{\A\Del_{\C}\A}(M\Del_{\C}N) \iso P\Del_{\A\Del_{\C}\A}((A_{(1)}\crta\ot M)\Del_{\C}(A_{(2)}\crta\ot N)).
\end{multline*}

\bigskip

Before we investigate associativity constraints for this newly defined tensor product in $\A\x\Mod$, let us develop some arithmetic in $\dul{\Pic}(\bullet)$
using \equref{M_i}. First of all, observe that given $\M\in\dul{\Pic}(\A^{\Del_{\C}n})$ we have: $\M_0=\A\Del_{\C}\M$ and $\M_{n+1}=\M\Del_{\C}\A$. Furthermore, given
$\N\in\dul{\Pic}(\A^{\Del_{\C}(n+1)}), \Ll\in\dul{\Pic}(\A^{\Del_{\C}(n+2)})$ and $i\in\{0,\dots, n+1\}, j\in\{0,\dots, n+2\}$, we find:
\begin{equation} \eqlabel{cuenta 1}
\M\Del_{\A^{\Del_{\C}n}}{}_{e^n_i}\N=\M\Del_{\A^{\Del_{\C} n}} {}_{e^n_i}(\A^{\Del_{\C} (n+1)}\Del_{\A^{\Del_{\C} (n+1)}}\N)={}_{e^n_i}\M_i\Del_{\A^{\Del_{\C} (n+1)}}\N
\end{equation}
as an $\A^{\Del_{\C} n}\x\A^{\Del_{\C} (n+1)}$-bimodule category. Similarly:
\begin{equation} \eqlabel{cuenta 2}
\M\Del_{\A^{\Del_{\C}n}}{}_{(e^{n+1}_j \comp e^n_i)}\Ll=\M\Del_{\A^{\Del_{\C} n}} {}_{(e^{n+1}_j \comp e^n_i)}(\A^{\Del_{\C} (n+2)}\Del_{\A^{\Del_{\C} (n+2)}}\Ll)=
{}_{(e^{n+1}_j \comp e^n_i)}\M_{ij}\Del_{\A^{\Del_{\C} (n+2)}}\Ll
\end{equation}
as an $\A^{\Del_{\C} n}\x\A^{\Del_{\C} (n+2)}$-bimodule category.

\bigskip

Now let $F\iso\Pp\Del_{\A\Del_{\C}\A}-$ be fixed for some invertible one-sided $\A\Del_{\C}\A$-bimodule category $\Pp$. Take $\M, \N, \Ll\in\A\x\Mod$.
Let us denote for simplicity reasons $\M\boxdot_F\N=\M\boxdot\N$. We find:
\begin{multline} \eqlabel{levi alfa}
(\M\boxdot\N)\boxdot\Ll
= {}_{\Delta}(\Pp\Del_{\A\Del_{\C}\A}({}_{\Delta}(\Pp\Del_{\A\Del_{\C}\A}(\M\Del_{\C}\N))\Del_{\C}\Ll))\\
= {}_{\Delta}(\Pp\Del_{\A\Del_{\C}\A}{}_{e_1^2=\Delta\Del_{\C}\Id}(\Pp_3\Del_{\A^{\Del_{\C}3}}(\M\Del_{\C}\N\Del_{\C}\Ll)))\\
\stackrel{\equref{cuenta 1}}{=} {}_{(\Delta\Del_{\C}\Id)\Delta}(\Pp_1\Del_{\A^{\Del_{\C}3}} \Pp_3\Del_{\A^{\Del_{\C}3}}(\M\Del_{\C}\N\Del_{\C}\Ll)).
\end{multline}
In order to compute $\M\boxdot(\N\boxdot\Ll)$, we first record the following, whose proof is direct:

\begin{lma}
There is an equivalence of left $\A\Del_{\C}\A$-module categories:
$$\M\Del_{\C} {}_{\Delta}(\Pp\Del_{\A\Del_{\C}\A}(\N\Del_{\C}\Ll)) \simeq {}_{e_2^2}(\Pp_0\Del_{\A^{\Del_{\C}3}}(\M\Del_{\C}\N\Del_{\C}\Ll))$$
given by (abusing the notation):
$$(A\crta\ot M)\Del_{\C}(P\Del_{\C}N\Del_{\C}L) \mapsfrom (A\Del_{\C}  P)\Del_{\A^{\Del_{\C}3}}(M\Del_{\C}N\Del_{\C}L)$$
for every $A\in\A, M\in\M, P\in\Pp, N\in\N, L\in\Ll$. 
\end{lma}

The above result can also be seen as a version of \cite[Proposition 6.6]{Gr1}. 

Now we have: 
\begin{multline} \eqlabel{desni alfa}
\M\boxdot(\N\boxdot\Ll)= {}_{\Delta}(\Pp\Del_{\A\Del_{\C}\A}(\M\Del_{\C} {}_{\Delta}(\Pp\Del_{\A\Del_{\C}\A}(\N\Del_{\C}\Ll))))\\
= {}_{\Delta}(\Pp\Del_{\A\Del_{\C}\A}{}_{e_2^2=\Id\Del_{\C}\Delta}(\Pp_0\Del_{\A^{\Del_{\C}3}}(\M\Del_{\C}\N\Del_{\C}\Ll)))\\
\stackrel{\equref{cuenta 1}}{=} {}_{(\Id\Del_{\C}\Delta)\Delta}(\Pp_2 \Del_{\A^{\Del_{\C}3}}\Pp_0\Del_{\A^{\Del_{\C}3}} (\M\Del_{\C}\N\Del_{\C}\Ll)).
\end{multline}
We get that a natural isomorphism in $\A\x\Mod$, that is a natural left $\A$-module equivalence
\begin{equation} \eqlabel{ass a}
a_{\M, \N, \Ll}: (\M\boxdot\N)\boxdot\Ll \stackrel{\simeq}{\to} \M\boxdot(\N\boxdot\Ll),
\end{equation}
corresponds to a natural left $\A$-module equivalence:
\begin{multline} \eqlabel{assoc}
\omega_{\M, \N, \Ll}: {}_{(\Delta\Del_{\C}\Id)\Delta}(\Pp_1\Del_{\A^{\Del_{\C}3}} \Pp_3\Del_{\A^{\Del_{\C}3}}(\M\Del_{\C}\N\Del_{\C}\Ll)) \\
 \stackrel{\simeq}{\to} {}_{(\Id\Del_{\C}\Delta)\Delta}(\Pp_2 \Del_{\A^{\Del_{\C}3}}\Pp_0\Del_{\A^{\Del_{\C}3}} (\M\Del_{\C}\N\Del_{\C}\Ll)).
\end{multline}

Let $\Pp, \Qq\in\Mod\x\A^{\Del_{\C} n}$ and consider two $n$-fold $\C$-balanced $n$-functors $\F, \G: \Mod\x\A\times\dots\times\Mod\x\A\to\Mod\x\A^{\Del_{\C} n}$ given by:
$$\F=\Pp\Del_{\A^{\Del_{\C} n}}(-\Del_{\C}-\cdots-\Del_{\C}-)\quad\textnormal{and}\quad \G=\Qq\Del_{\A^{\Del_{\C} n}}(-\Del_{\C}-\cdots-\Del_{\C}-).$$
Denote $\pi^n_{\C}=-\Del_{\C}-\cdots-\Del_{\C}-: \Mod\x\A\times\dots\times\Mod\x\A\to\Mod\x\A^{\Del_{\C} n}$,
which is an $n$-fold $\C$-balanced $n$-functor. Then we may write $\F=(\Pp\Del_{\A^{\Del_{\C} n}}-)\comp\pi^n_{\C}$ and $\G=(\Qq\Del_{\A^{\Del_{\C} n}}-)\comp\pi^n_{\C}$.
Recall that $\Nat(\F,\G)$ denotes the class of natural transformations from $\F$ to $\G$.

\begin{lma} \lelabel{EW morfisms}
Let $\A$ be a symmetric $\C$-bialgebroid category and let $\Pp, \Qq\in\Mod\x\A^{\Del_{\C} n}$.
There are isomorphisms:
\begin{eqnarray*}
\Nat((\Pp\Del_{\A^{\Del_{\C} n}}-)\comp\pi^n_{\C}, (\Qq\Del_{\A^{\Del_{\C} n}}-)\comp\pi^n_{\C}) &\iso & \Nat(\Pp\Del_{\A^{\Del_{\C} n}}-, \Qq\Del_{\A^{\Del_{\C} n}}-)\\
&\iso& \Fun(\Pp,\Qq)_{\A^{\Del_{\C} n}}.
\end{eqnarray*}
\end{lma}

\begin{proof}
The second and the third class are in one-to-one bijection by \thref{EW} applied on the level of morphisms (which refer to the objects
$\Pp, \Qq\in\Mod\x\A^{\Del_{\C} n}$). Observe that we consider all bimodule structures as one-sided.
Given $\omega\in \Nat(\Pp\Del_{\A^{\Del_{\C} n}}-, \Qq\Del_{\A^{\Del_{\C} n}}-)$ it induces
$\tilde\omega\in\Nat((\Pp\Del_{\A^{\Del_{\C} n}}-)\comp\pi^n_{\C}, (\Qq\Del_{\A^{\Del_{\C} n}}-)\comp\pi^n_{\C})$ by defining
$\tilde\omega(\M_1, \dots, \M_n)=\omega(\M_1\Del_{\C}\cdots\Del_{\C}\M_n)$ and $\tilde\omega(F_1, \dots, F_n)=\omega(F_1\Del_{\C}\cdots\Del_{\C}F_n)$
for $F_i:\M_i\to\N_i$ in $\A\x\Mod$ and $i=1,\dots, n$.
Conversely, given $\omega\in\Nat((\Pp\Del_{\A^{\Del_{\C} n}}-)\comp\pi^n_{\C}, (\Qq\Del_{\A^{\Del_{\C} n}}-)\comp\pi^n_{\C})$ consider the
left $\A$-module functors $F_i=-\crta\ot M_i:\A\to\M_i$ for fixed objects $M_i\in\M_i$ and $\M_i\in\A\x\Mod$ for $i=1,\dots, n$. Now proceed
as in the proof of \thref{pre-EW} to prove that $\omega(\A, \dots, \A)$ is a right $\A^{\Del_{\C} n}$-module functor. Namely, the naturality of $\omega$ implies that
we have a commutative diagram:
\begin{equation*}
\scalebox{0.88}{\bfig
\putmorphism(-180,400)(1,0)[\Pp` \Qq `\omega_{\A, \dots, \A}=\alpha]{2020}1a
\putmorphism(-350,0)(1,0)[\Pp\Del_{\A^{\Del_{\C} n}}(\M_1\Del_{\C}\cdots\Del_{\C} \M_n)` \Qq\Del_{\A^{\Del_{\C} n}}(\M_1\Del_{\C}\cdots\Del_{\C} \M_n),` \omega_{\M_1, \dots, \M_n}]{2000}1a
\putmorphism(-160,400)(0,-1)[\phantom{B}``\Pp\Del_{\A^{\Del_{\C} n}}(F_1\Del_{\C}\cdots\Del_{\C} F_n)]{380}1l
\putmorphism(1850,400)(0,-1)[\phantom{B\ot B}``\Qq\Del_{\A^{\Del_{\C} n}}(F_1\Del_{\C}\cdots\Del_{\C} F_n)]{380}1r
\efig}
\end{equation*}
which applied to $P\in\Pp$ yields:
\begin{equation} \eqlabel{omega i alfa strict}
\omega_{\M_1, \dots, \M_n}(P\Del_{\A^{\Del_{\C} n}}(M_1\Del_{\C}\cdots\Del_{\C} M_n))=\alpha(P)\Del_{\A^{\Del_{\C} n}}(M_1\Del_{\C}\cdots\Del_{\C} M_n).
\end{equation}
Setting $\M_i=\A$ for every $i=1,\dots, n$ we get that $\alpha$ is a strict right $\A^{\Del_{\C} n}$-module functor. As we saw in \thref{EW} when proving the fullness
of the functor $\HH$, by \cite[Proposition 2.8]{Ga1} the ``strictness'' is irrelevant. Finally, note that the two assignments are inverse to
each other and we have the claim.
\qed\end{proof}

By this lemma, the natural left $\A$-module equivalence $\omega_{\M, \N, \Ll}$ from \equref{assoc}
is determined by an $\A\x\A^{\Del_{\C} 3}$-bimodule equivalence functor $\alpha=\omega_{\A, \A, \A}$:
\begin{equation} \eqlabel{alfa}
\alpha: {}_{(\Delta\Del_{\C}\Id)\Delta}(\Pp_1\Del_{\A^{\Del_{\C}3}} \Pp_3) \to
{}_{(\Id\Del_{\C}\Delta)\Delta}(\Pp_2 \Del_{\A^{\Del_{\C}3}}\Pp_0).
\end{equation}
By \equref{omega i alfa strict} we have:
\begin{equation} \eqlabel{omega-alfa}
\omega_{\M, \N, \Ll}=\alpha\Del_{\A^{\Del_{\C} 3}}\Id_{\M\Del_{\C}\N\Del_{\C}\Ll}.
\end{equation}

Let us now consider the equivalence functor:
\begin{equation} \eqlabel{beta}
\beta=(\crta{ev}_1\Del_{\A^{\Del_{\C}3}} ev_3)(\Pp_1^{op}\Del_{\A^{\Del_{\C}3}}\alpha^{-1}\Del_{\A^{\Del_{\C}3}}\Pp_3^{op}):
\Pp_1^{op}\Del_{\A^{\Del_{\C}3}}\Pp_2\Del_{\A^{\Del_{\C}3}}\Pp_0\Del_{\A^{\Del_{\C}3}}\Pp_3^{op}\to\A^{\Del_{\C}3}.
\end{equation}
At the end of \ssref{1st subsection} we recalled that for invertible bimodule categories the evaluation and coevaluation functors are equivalences and that they are inverse to each other.

\begin{thm} \thlabel{asoc-coc}
Let $\Pp\in\dul{\Pic}(\A\Del_{\C}\A), F=\Pp\Del_{\A\Del_{\C}\A}-$ and let $\alpha$ be the natural equivalence \equref{alfa}.
The associativity constraint for the tensor product functor defined by \equref{tensor functor}
satisfies the pentagonal axiom up to a natural equivalence if and only if $(\Pp, \beta)\in\dul{Z}^2(\A,\dul{\Pic})$, where $\beta$ is the equivalence given by \equref{beta}.
\end{thm}

\begin{proof}
First let us compute the new tensor products in the pentagonal diagram:
\begin{equation} \eqlabel{pent dot}
\scalebox{0.84}{
\bfig 
\putmorphism(-200,500)(1,0)[((\M\boxdot\N)\boxdot\Ll)\boxdot\Qq` (\M\boxdot(\N\boxdot\Ll))\boxdot\Qq ` a_{\M, \N, \Ll}\boxdot\Qq]{1600}1a
\putmorphism(1400,500)(1,0)[\phantom{(X \ot (Y \ot U)) \ot W}`\M\boxdot((\N\boxdot\Ll)\boxdot\Qq)` a_{\M, \N\boxdot\Ll, \Qq}]{1540}1a
\putmorphism(3160,500)(0,-1)[``\M\boxdot a_{\N,\Ll, \Qq}]{500}1l
\putmorphism(-160,500)(0,-1)[``a_{\M\boxdot\N,\Ll, \Qq}]{500}1r
\putmorphism(-200,0)(1,0)[(\M\boxdot\N)\boxdot(\Ll\boxdot\Qq)` \M\boxdot((\N\boxdot\Ll)\boxdot\Qq).` a_{\M, \N,\Ll\boxdot\Qq}]{3180}1b
\efig}
\end{equation}
Observe that $a$ is the natural equivalence from \equref{ass a} which we identify with the equivalence $\omega$ from \equref{assoc}, which in turn is related to $\alpha$ by
\equref{omega-alfa}. The five vertices in the above diagram can be rewritten as follows:

\begin{multline*}  
((\M\boxdot\N)\boxdot\Ll)\boxdot\Qq \stackrel{\equref{levi alfa}}{=}
{}_{(\Delta\Del_{\C}\Id)\Delta}(\Pp_1\Del_{\A^{\Del_{\C}3}} \Pp_3\Del_{\A^{\Del_{\C}3}}(\M\Del_{\C}\N\Del_{\C}\Ll))\boxdot\Qq\\
\stackrel{\equref{new tensor}}{=} {}_{\Delta}\Pp\Del_{\A\Del_{\C}\A} ( {}_{(\Delta\Del_{\C}\Id)\Delta}(\Pp_1\Del_{\A^{\Del_{\C}3}} \Pp_3\Del_{\A^{\Del_{\C}3}}(\M\Del_{\C}\N\Del_{\C}\Ll))\Del_{\C}\Qq )\\
= {}_{\Delta}\Pp\Del_{\A\Del_{\C}\A} ( {}_{(\Delta\Del_{\C}\Id\Del_{\C}\Id)(\Delta\Del_{\C}\Id)}(\Pp_{14}\Del_{\A^{\Del_{\C}4}} \Pp_{34}\Del_{\A^{\Del_{\C}4}}(\M\Del_{\C}\N\Del_{\C}\Ll\Del_{\C}\Qq)))\\
\stackrel{\equref{cuenta 2}}{=}
{}_{(\Delta\Del_{\C}\Id\Del_{\C}\Id)(\Delta\Del_{\C}\Id)\Delta} \Pp_{11}\Del_{\A^{\Del_{\C}4}} (\Pp_1\Del_{\A^{\Del_{\C}3}} \Pp_3)_4 \Del_{\A^{\Del_{\C}4}} (\M\Del_{\C}\N\Del_{\C}\Ll\Del_{\C}\Qq)\\
\stackrel{14=31}{=}
{}_{(\Delta\Del_{\C}\Id\Del_{\C}\Id)(\Delta\Del_{\C}\Id)\Delta} (\Pp_1\Del_{\A^{\Del_{\C}3}} \Pp_3)_1\Del_{\A^{\Del_{\C}4}} \Pp_{34}\Del_{\A^{\Del_{\C}4}} (\M\Del_{\C}\N\Del_{\C}\Ll\Del_{\C}\Qq)\\
\end{multline*}

\begin{multline*}  
(\M\boxdot(\N\boxdot\Ll))\boxdot\Qq \stackrel{\equref{levi alfa},\equref{new tensor}}{=}
{}_{(\Delta\Del_{\C}\Id)\Delta}(\Pp_1\Del_{\A^{\Del_{\C}3}} \Pp_3\Del_{\A^{\Del_{\C}3}}(\M\Del_{\C}({}_{\Delta}\Pp\Del_{\A\Del_{\C}\A}(\N\Del_{\C}\Ll))\Del_{\C}\Qq))\\
={}_{(\Delta\Del_{\C}\Id)\Delta}(\Pp_1\Del_{\A^{\Del_{\C}3}} \Pp_3 \Del_{\A^{\Del_{\C}3}}
   {}_{(\Id\Del_{\C}\Delta\Del_{\C}\Id)}(\Pp_{04}\Del_{\A^{\Del_{\C}4}}(\M\Del_{\C}\N\Del_{\C}\Ll\Del_{\C}\Qq)))\\
\stackrel{\equref{cuenta 1}}{=}
{}_{(\Id\Del_{\C}\Delta\Del_{\C}\Id)(\Delta\Del_{\C}\Id)\Delta)} ((\Pp_1\Del_{\A^{\Del_{\C}3}} \Pp_3)_2\Del_{\A^{\Del_{\C}4}} \Pp_{04}\Del_{\A^{\Del_{\C}4}} (\M\Del_{\C}\N\Del_{\C}\Ll\Del_{\C}\Qq))\\
\stackrel{24=32}{=}
{}_{(\Id\Del_{\C}\Delta\Del_{\C}\Id)(\Delta\Del_{\C}\Id)\Delta} (\Pp_{12}\Del_{\A^{\Del_{\C}4}} (\Pp_2\Del_{\A^{\Del_{\C}3}} \Pp_0)_4\Del_{\A^{\Del_{\C}4}} (\M\Del_{\C}\N\Del_{\C}\Ll\Del_{\C}\Qq))\\
\end{multline*}

\begin{multline*}  
\M\boxdot((\N\boxdot\Ll)\boxdot\Qq) \stackrel{\equref{levi alfa},\equref{new tensor}}{=}
{}_{\Delta}(\Pp\Del_{\A\Del_{\C}\A}(\M\Del_{\C} {}_{(\Delta\Del_{\C}\Id)\Delta}(\Pp_1\Del_{\A^{\Del_{\C}3}}\Pp_3 \Del_{\A^{\Del_{\C}3}} (\N\Del_{\C}\Ll\Del_{\C}\Qq)))\\
={}_{\Delta}(\Pp\Del_{\A\Del_{\C}\A} {}_{(\Id\Del_{\C}\Delta\Del_{\C}\Id)(\Id\Del_{\C}\Delta)} ((\Pp_1\Del_{\A^{\Del_{\C}3}} \Pp_3)_0\Del_{\A^{\Del_{\C}4}}(\M\Del_{\C}\N\Del_{\C}\Ll\Del_{\C}\Qq)))\\
\stackrel{\equref{cuenta 2}}{=}
{}_{(\Id\Del_{\C}\Delta\Del_{\C}\Id)(\Id\Del_{\C}\Delta)\Delta} (\Pp_{22}\Del_{\A^{\Del_{\C}4}} (\Pp_1\Del_{\A^{\Del_{\C}3}} \Pp_3)_0\Del_{\A^{\Del_{\C}4}} (\M\Del_{\C}\N\Del_{\C}\Ll\Del_{\C}\Qq))\\
\stackrel{10=02}{\stackrel{30=04}{=} }
{}_{(\Id\Del_{\C}\Delta\Del_{\C}\Id)(\Id\Del_{\C}\Delta)\Delta} ((\Pp_2\Del_{\A^{\Del_{\C}3}}\Pp_0)_2\Del_{\A^{\Del_{\C}4}} \Pp_{04}\Del_{\A^{\Del_{\C}4}} (\M\Del_{\C}\N\Del_{\C}\Ll\Del_{\C}\Qq))\\
\end{multline*}

\begin{multline*}  
\M\boxdot(\N\boxdot(\Ll\boxdot\Qq)) \stackrel{\equref{desni alfa},\equref{new tensor}}{=}
{}_{\Delta}(\Pp\Del_{\A\Del_{\C}\A}(\M\Del_{\C} {}_{(\Id\Del_{\C}\Delta)\Delta}(\Pp_2\Del_{\A^{\Del_{\C}3}}\Pp_0 \Del_{\A^{\Del_{\C}3}} (\N\Del_{\C}\Ll\Del_{\C}\Qq)))\\
={}_{\Delta}(\Pp\Del_{\A\Del_{\C}\A} {}_{(\Id\Del_{\C}\Id\Del_{\C}\Delta)(\Id\Del_{\C}\Delta)} ((\Pp_2\Del_{\A^{\Del_{\C}3}} \Pp_0)_0\Del_{\A^{\Del_{\C}4}}(\M\Del_{\C}\N\Del_{\C}\Ll\Del_{\C}\Qq)))\\
\stackrel{\equref{cuenta 2}}{=}
{}_{(\Id\Del_{\C}\Id\Del_{\C}\Delta)(\Id\Del_{\C}\Delta)\Delta} (\Pp_{23}\Del_{\A^{\Del_{\C}4}} (\Pp_2\Del_{\A^{\Del_{\C}3}} \Pp_0)_0\Del_{\A^{\Del_{\C}4}} (\M\Del_{\C}\N\Del_{\C}\Ll\Del_{\C}\Qq))\\
\stackrel{20=03}{=}
{}_{(\Delta\Del_{\C}\Id\Del_{\C}\Id)(\Delta\Del_{\C}\Id)\Delta} ((\Pp_2\Del_{\A^{\Del_{\C}3}}\Pp_0)_3\Del_{\A^{\Del_{\C}4}} \Pp_{00}\Del_{\A^{\Del_{\C}4}} (\M\Del_{\C}\N\Del_{\C}\Ll\Del_{\C}\Qq))\\
\end{multline*}

\begin{multline*}  
(\M\boxdot\N)\boxdot(\Ll\boxdot\Qq) \stackrel{\equref{new tensor}}{=}
{}_{\Delta}(\Pp\Del_{\A\Del_{\C}\A} \left( {}_{\Delta}(\Pp\Del_{\A\Del_{\C}\A} (\M\Del_{\C}\N)) \Del_{\C} {}_{\Delta}(\Pp\Del_{\A^{\Del_{\C}2}} (\Ll\Del_{\C}\Qq)) \right) \\
\stackrel{\equref{cuenta 2}}{=}
{}_{\Delta}\Pp\Del_{\A\Del_{\C}\A} ( {}_{(e^3_4\comp e^2_3)\Delta}(\Pp_{34}) \Del_{\C} {}_{(e^3_0\comp e^2_0)\Delta}(\Pp_{00})) \Del_{\A^{\Del_{\C}4}}(\M\Del_{\C}\N\Del_{\C}\Ll\Del_{\C}\Qq))\\
={}_{\Delta}\Pp\Del_{\A\Del_{\C}\A} {}_{(\Delta\Del_{\C}\Delta)}(\Pp_{34}  \Del_{\A^{\Del_{\C}4}} \Pp_{00} \Del_{\A^{\Del_{\C}4}}(\M\Del_{\C}\N\Del_{\C}\Ll\Del_{\C}\Qq)))\\
\stackrel{\equref{cuenta 2}, 34=33}{=}
{}_{(\Delta\Del_{\C}\Delta)\Delta} \Pp_{13}\Del_{\A^{\Del_{\C}4}} \Pp_{33} \Del_{\A^{\Del_{\C}4}} \Pp_{00} \Del_{\A^{\Del_{\C}4}} (\M\Del_{\C}\N\Del_{\C}\Ll\Del_{\C}\Qq)\\
=
{}_{(\Delta\Del_{\C}\Delta)\Delta} (\Pp_1\Del_{\A^{\Del_{\C}3}} \Pp_3)_3\Del_{\A^{\Del_{\C}4}} \Pp_{00}\Del_{\A^{\Del_{\C}4}} (\M\Del_{\C}\N\Del_{\C}\Ll\Del_{\C}\Qq)\\
\stackrel{\Id\Del_4\tau}{=}
{}_{(\Delta\Del_{\C}\Delta)\Delta} (\Pp_{13}\Del_{\A^{\Del_{\C}4}} \Pp_{00} \Del_{\A^{\Del_{\C}4}} \Pp_{33} \Del_{\A^{\Del_{\C}4}} (\M\Del_{\C}\N\Del_{\C}\Ll\Del_{\C}\Qq)\\
\stackrel{13=21}{\stackrel{00=01}{\stackrel{33=34}{=}}}
{}_{(\Delta\Del_{\C}\Delta)\Delta} (\Pp_2\Del_{\A^{\Del_{\C}3}} \Pp_0)_1 \Del_{\A^{\Del_{\C}4}} \Pp_{34} \Del_{\A^{\Del_{\C}4}} (\M\Del_{\C}\N\Del_{\C}\Ll\Del_{\C}\Qq)\\
\end{multline*}

In the last computation we used:
\begin{multline*}
((e^3_4\comp e^2_3)\Delta \Del_{\A^{\Del_{\C}4}} (e^3_0\comp e^2_0)\Delta)(A\Del_{\C}B)=(\Delta(A)\Del_{\C}I\Del_{\C}I)\Del_{\A^{\Del_{\C}4}}(I\Del_{\C}I\Del_{\C}\Delta(B))\\
=\Delta(A)\Del_{\C}\Delta(B).
\end{multline*}

Observe that by \leref{EW morfisms} the diagram \equref{pent dot} commutes if and only if so does the diagram:
\begin{equation} \eqlabel{pent P}
\scalebox{0.84}{
\bfig 
\putmorphism(-200,500)(1,0)[\Pp_{11}\Del_{\A^{\Del_{\C}4}} (\Pp_1\Del_{\A^{\Del_{\C}3}} \Pp_3)_4 ` \Pp_{12}\Del_{\A^{\Del_{\C}4}} (\Pp_2\Del_{\A^{\Del_{\C}4}} \Pp_0)_4 `
  \Pp_{12}\Del_{\A^{\Del_{\C}4}}\alpha_4]{1800}1a
\putmorphism(1600,500)(0,-1)[\phantom{(X \ot (Y \ot U)) \ot W}` (\Pp_1\Del_{\A^{\Del_{\C}3}} \Pp_3)_2\Del_{\A^{\Del_{\C}4}} \Pp_{04}` =]{250}1l
\putmorphism(1700,280)(1,0)[\phantom{(X \ot (Y \ot U)) \ot W}`(\Pp_2\Del_{\A^{\Del_{\C}3}}\Pp_0)_2\Del_{\A^{\Del_{\C}4}} \Pp_{04}` \alpha_2\Del_{\A^{\Del_{\C}4}}\Pp_{04}]{1600}1a
\putmorphism(3240,300)(0,-1)[`\Pp_{23}\Del_{\A^{\Del_{\C}4}} (\Pp_1\Del_{\A^{\Del_{\C}3}} \Pp_3)_0`=]{250}1l
\putmorphism(3240,60)(0,-1)[`\Pp_{23}\Del_{\A^{\Del_{\C}4}} (\Pp_2\Del_{\A^{\Del_{\C}3}} \Pp_0)_0 ` \Pp_{23}\Del_{\A^{\Del_{\C}4}}\alpha_0]{450}1l
\putmorphism(3240,-390)(0,-1)[\phantom{(X \ot (Y \ot U)) \ot W}` ` =]{260}1l
\putmorphism(-200,500)(0,-1)[\phantom{(X \ot (Y \ot U)) \ot W}` (\Pp_1\Del_{\A^{\Del_{\C}3}} \Pp_3)_1\Del_{\A^{\Del_{\C}4}} \Pp_{34}` =]{300}1r
\putmorphism(-200,190)(0,-1)[\phantom{(X \ot (Y \ot U)) \ot W}` (\Pp_2\Del_{\A^{\Del_{\C}3}} \Pp_0)_1 \Del_{\A^{\Del_{\C}4}} \Pp_{34}` \alpha_1\Del_{\A^{\Del_{\C}4}}\Pp_{34}]{470}1r
\putmorphism(-200,-290)(0,-1)[\phantom{(X \ot (Y \ot U)) \ot W}` (\Pp_1\Del_{\A^{\Del_{\C}3}} \Pp_3)_3\Del_{\A^{\Del_{\C}4}} \Pp_{00}` \Id\Del_{\A^{\Del_{\C}4}}\tau]{360}1r
\putmorphism(-80,-660)(1,0)[\phantom{(X \ot (Y \ot U)) \ot W}` (\Pp_2\Del_{\A^{\Del_{\C}3}}\Pp_0)_3\Del_{\A^{\Del_{\C}4}} \Pp_{00}` \alpha_3\Del_{\A^{\Del_{\C}4}}\Pp_{00}]{3320}1b
\efig}
\end{equation}
This means that there is a natural equivalence between the functors:
\begin{equation} \eqlabel{ident}
(\alpha_1^{-1}\Del_{\A^{\Del_{\C}4}}\Pp_{34})(\Id\Del_{\A^{\Del_{\C}4}}\tau)(\alpha_3^{-1}\Del_{\A^{\Del_{\C}4}}\Pp_{00})(\Pp_{23}\Del_{\A^{\Del_{\C}4}}\alpha_0)
(\alpha_2\Del_{\A^{\Del_{\C}4}}\Pp_{04})(\Pp_{12}\Del_{\A^{\Del_{\C}4}}\alpha_4)\simeq\Id
\end{equation}
acting on $D_0=\Pp_{11}\Del_{\A^{\Del_{\C}4}}\Pp_{14}\Del_{\A^{\Del_{\C}4}}\Pp_{34}$.
Set $\F_1, \dots, \F_6$ for the functors on the left hand-side in \equref{ident} 
reading from the right to the left, and set $D_0, D_1, \dots, D_5$ for their respective domains. The latter are objects in the six vertices of the diagram
\equref{pent P} in the clockwise order (we identify the vertices which have equal objects). 
Moreover, set $D=D_1\Del_{\A^{\Del_{\C}4}} D_2\Del_{\A^{\Del_{\C}4}}\dots \Del_{\A^{\Del_{\C}4}}D_5$. By \equref{comp-tensor} from \coref{comp-tensor} we have:
$$(D\Del_{\A^{\Del_{\C}4}}(\F_6\comp\F_5\comp\dots\comp\F_1))\comp\tau_{D_0,D}=\F_1\Del_{\A^{\Del_{\C}4}}\dots\Del_{\A^{\Del_{\C}4}}\F_6.$$
Then by \leref{basic lema}, 1) the identity \equref{ident} is equivalent to having a natural equivalence between the functors:
$$\F_1\Del_{\A^{\Del_{\C}4}}\dots\Del_{\A^{\Del_{\C}4}}\F_6\simeq\Id_{ _{D_0\Del_{\A^{\Del_{\C}4}}D}}.$$
Up to reordering of the factors (in $\dul{\Pic}(\A^{\Del_{\C}4})$) and again by \leref{basic lema}, 1) we get equivalently that
there is a natural equivalence:
\begin{equation} \eqlabel{alfa's cond}
\alpha_1^{-1}\Del_{\A^{\Del_{\C}4}}\alpha_3^{-1}\Del_{\A^{\Del_{\C}4}}\alpha_0\Del_{\A^{\Del_{\C}4}}\alpha_2\Del_{\A^{\Del_{\C}4}}\alpha_4\simeq\Id_{ _{}}
\end{equation}
on the category obtained by canceling out the factors $\Pp_{34}, \Pp_{00}, \Pp_{23}, \Pp_{04}, \Pp_{12}$ from $D_0\Del_{\A^{\Del_{\C}4}}D$. That is
the product in $\dul{\Pic}(\A^{\Del_{\C}4})$ of the categories labeled by:
$$01, \quad 02, \quad 03, \quad 04, \quad 12, \quad 13, \quad 14, \quad 23, \quad 24, \quad 34$$
in the increasing order of indices.

\bigskip

On the other hand, we have $(\Pp, \beta)\in\dul{Z}^2(\A,\dul{\Pic})$ if and only if $\delta_3(\beta)\simeq\lambda_{\Pp}$, which by \leref{alfa dagger}, 2) we can write as:
\begin{equation}\eqlabel{cocycle cond}
\end{equation}
$$\begin{array}{ccr}
\hspace{-0,6cm}
\beta_0\Del_{\A^{\Del_{\C} 4}}\beta_1^{\dagger}\Del_{\A^{\Del_{\C} 4}}\beta_2\Del_{\A^{\Del_{\C} 4}}\beta_3^{\dagger}\Del_{\A^{\Del_{\C} 4}}\beta_4 &\simeq&
\ev_{01}\Del_{\A^{\Del_{\C} 4}}\ev_{02}\Del_{\A^{\Del_{\C} 4}}\ev_{03}\Del_{\A^{\Del_{\C} 4}}\ev_{04}\\
& & \Del_{\A^{\Del_{\C} 4}}\ev_{12}\Del_{\A^{\Del_{\C} 4}}\ev_{13}\Del_{\A^{\Del_{\C} 4}}\ev_{14}\\
& & \Del_{\A^{\Del_{\C} 4}}\ev_{23}\Del_{\A^{\Del_{\C} 4}}\ev_{24}\\
& & \Del_{\A^{\Del_{\C} 4}}\ev_{34}.
\end{array}$$
Recall that the equivalence functor $\beta_i^{\dagger}: \delta_2(\Pp)^{op}_i\to\A^{\Del_{\C} 4}$ is given by:
\begin{equation}\eqlabel{alfa dagger i}
\beta_i^{\dagger}=
\crta\ev_{\delta_2(\Pp)_i}(\Id_{\delta_2(\Pp)_i^{op}} \Del_{\A^{\Del_{\C} 4}}\beta^{-1}_i).
\end{equation}
Observe that $\beta^{-1}=(\Pp_1^{op}\Del_{\A^{\Del_{\C}3}}\alpha\Del_{\A^{\Del_{\C}3}}\Pp_3^{op})(coev_1\Del_{\A^{\Del_{\C}3}} \crta{coev}_3)$ and consider
$\delta_2(\Pp)=\Pp_1^{op}\Del_{\A^{\Del_{\C}3}}\Pp_2\Del_{\A^{\Del_{\C}3}}\Pp_0\Del_{\A^{\Del_{\C}3}}\Pp_3^{op}$. Then we may write in braided diagrams - we will
use the notation $ij$ for $\Pp_{ij}$ and $\crta{ij}$ for $\Pp_{ij}^{op}$:
$$\beta_i=
\gbeg{5}{4}
\got{1}{\crta{1i}} \got{1}{2i} \got{1}{0i} \got{1}{\crta{3i}} \gnl
\gcl{1} \glmptb\gnot{\hspace{-0,4cm}\alpha_i^{-1}}\grmptb \gcl{1} \gnl
\gev \gev \gnl
\gob{1}{}
\gend
\quad\textnormal{and}\quad\quad
\beta_i^{\dagger}=
\gbeg{8}{7}
\got{1}{1i} \got{1}{\crta{2i}} \got{1}{\crta{0i}} \got{1}{3i} \gnl
\gcl{5} \gcl{4} \gcl{3} \gcl{2} \gdb \gdb \gnl
\gvac{4} \gcl{1} \glmptb\gnot{\hspace{-0,4cm}\alpha_i}\grmptb \gcl{1} \gnl
\gvac{3} \gbr \gcl{1} \gcl{2} \gcl{3} \gnl
\gvac{2} \gbr \gbr \gnl
\gvac{1} \gbr \gbr \gbr \gcl{1} \gnl
\gev \gev \gev \gev \gnl
\gob{1}{}
\gend=
\gbeg{6}{4}
\got{1}{1i} \got{1}{\crta{2i}} \got{1}{\crta{0i}} \got{1}{3i} \gnl
\gibr \gibr \gnl
\gcl{1} \glmptb\gnot{\hspace{-0,4cm}\alpha_i}\grmptb \gcl{1} \gnl
\gev
\gev \gnl
\gob{1}{}
\gend
$$
where we applied naturality and the dual basis axioms, identifying $ev$ and $\crta{ev}=ev\comp\tau$. In other words, we have:
$$\beta_i=(\crta{ev}_{1i}\Del_{\A^{\Del_{\C}4}}ev_{3i})(\Pp_{1i}^{op}\Del_{\A^{\Del_{\C}4}}\alpha_i^{-1}\Del_{\A^{\Del_{\C}4}}\Pp_{3i}^{op})$$
and 
$$\beta_i^{\dagger}=(\crta{ev}_{2i}\Del_{\A^{\Del_{\C}4}}ev_{0i})(\Pp_{2i}^{op}\Del_{\A^{\Del_{\C}4}}\alpha_i\Del_{\A^{\Del_{\C}4}}\Pp_{0i}^{op}).$$
Then the left hand-side of \equref{cocycle cond} is naturally equivalent to:
$$\Ev\comp(\Id_{ _{}}\Del_{\A^{\Del_{\C}4}}\alpha_0^{-1}\Del_{\A^{\Del_{\C}4}}\alpha_1\Del_{\A^{\Del_{\C}4}}\alpha_2^{-1}\Del_{\A^{\Del_{\C}4}}\alpha_3\Del_{\A^{\Del_{\C}4}}\alpha_4^{-1})$$
where $\Ev$ is the tensor product in $\dul{\Pic}(\A^{\Del_{\C}4})$ of the evaluations labeled by:
$$10=02, \quad 30=04, \quad 21=13, \quad 01, \quad 12, \quad 32=24, \quad 23, \quad 03, \quad 14, \quad 34$$
and $\Id$ is the identity functor on the tensor product of the corresponding $\Pp_{ij}'s$. Note that up to reordering of the factors it is: $\Ev\simeq \lambda_{\Pp}$. Then
the equation \equref{cocycle cond} becomes:
$$\lambda_{\Pp}\comp(\Id_{ _{}}\Del_{\A^{\Del_{\C}4}}\alpha_0^{-1}\Del_{\A^{\Del_{\C}4}}\alpha_1\Del_{\A^{\Del_{\C}4}}\alpha_2^{-1}\Del_{\A^{\Del_{\C}4}}\alpha_3\Del_{\A^{\Del_{\C}4}}\alpha_4^{-1})
\simeq\lambda_{\Pp}$$
which is equivalent to:
$$\Id_{ _{}}\Del_{\A^{\Del_{\C}4}}\alpha_0^{-1}\Del_{\A^{\Del_{\C}4}}\alpha_1\Del_{\A^{\Del_{\C}4}}\alpha_2^{-1}\Del_{\A^{\Del_{\C}4}}\alpha_3\Del_{\A^{\Del_{\C}4}}\alpha_4^{-1}
\simeq\Id$$ 
for the corresponding identity functor on the right hand-side. By \leref{basic lema}, 1) we get equivalently equation \equref{alfa's cond}.
\qed\end{proof}

\bigskip

We may organize quasi-monoidal structures on $\A\x\Mod$ of the form \equref{tensor functor} into a category, so that this category is equivalent to the category
of 2-cocycles $\dul{Z}^2(\A,\dul{\Pic})$. Let $\dul{QMon}(\A)$ be the category whose objects are quasi-monoidal structures $(\A\x\Mod, \boxdot, a)$ of the form
\equref{tensor functor} and morphisms are quasi-monoidal 2-functors between them, defined as follows. A {\em quasi-monoidal 2-functor}
$(\Id, \Psi): (\A\x\Mod, \boxdot, a)\to(\A\x\Mod, \boxdot', a')$ is the identity functor  $\Id$ on $\A\x\Mod$ together with its quasi-monoidal structure
$\Psi: \M\boxdot\N=\Id(\M\boxdot\N)\to\Id(\M)\boxdot'\Id(\N)=\M\boxdot'\N$, for $\M, \N\in\A\x\Mod,$ which is an equivalence in $\A\x\Mod$, so that the
hexagonal coherence for its action on three objects holds up to a natural equivalence. A ``quasi-monoidal'' hear refers to the compatibility only with the tensor product.

\begin{thm} \thlabel{monoidal equiv}
There is an equivalence of categories:
$$\dul{QMon}(\A)\simeq\dul{Z}^2(\A,\dul{\Pic}).$$
\end{thm}

\begin{proof}
Given a quasi-monoidal structure $(\A\x\Mod, \boxdot, a)$, set $\Pp=\A\boxdot\A$. We know by \thref{asoc-coc} that $(\Pp, \beta)\in\dul{Z}^2(\A,\dul{\Pic})$, where
$\beta$ is given by \equref{beta} and $\alpha=
a_{\A,\A,\A}$. Let now $(\Id, \Psi)$ be a morphism between two quasi-monoidal structures $(\A\x\Mod, \boxdot, a)$ and $(\A\x\Mod, \boxdot', a')$. Then for every
$\M,\N\in\A\x\Mod$ we have an equivalence $\Psi_{\M,\N}: \M\boxdot\N\to\M\boxdot'\N$ in $\A\x\Mod$, so that the following diagram commutes up to a natural equivalence:
\begin{equation} \eqlabel{Psi na prod}
\scalebox{0.84}{
\bfig 
\putmorphism(-200,500)(1,0)[(\M\boxdot\N)\boxdot\Ll` \M\boxdot(\N\boxdot\Ll) ` a_{\M, \N, \Ll}]{1600}1a
\putmorphism(1400,500)(1,0)[\phantom{(X \ot (Y \ot U))}` \M\boxdot(\N\boxdot'\Ll)` \M\boxdot\Psi_{\N, \Ll}]{1540}1a
\putmorphism(3160,500)(0,-1)[``\Psi_{\M,\N\boxdot'\Ll}]{500}1l
\putmorphism(-160,500)(0,-1)[``\Psi_{\M,\N}\boxdot\Ll]{500}1r
\putmorphism(-200,0)(1,0)[(\M\boxdot'\N)\boxdot\Ll` (\M\boxdot'\N)\boxdot'\Ll` \Psi_{\M\boxdot'\N,\Ll}]{1600}1b
\putmorphism(1400,0)(1,0)[\phantom{(X \ot (Y \ot U))}`  \M\boxdot'(\N\boxdot'\Ll)` a'_{\M, \N, \Ll}]{1600}1b
\efig}
\end{equation}
Set $F=\Psi_{\A,\A}:\Pp\to\Pp'$. We know from \leref{EW morfisms} and the calculations of \equref{levi alfa} and \equref{desni alfa} that the above diagram commutes
if and only if so does the diagram:
\begin{equation} \eqlabel{alfa na P}
\scalebox{0.84}{
\bfig 
\putmorphism(-200,500)(1,0)[\Pp_1\Del_{\A^{\Del_{\C}3}}\Pp_3` \Pp_2\Del_{\A^{\Del_{\C}3}}\Pp_0 ` \alpha]{1600}1a
\putmorphism(1400,500)(1,0)[\phantom{(X \ot (Y \ot U))}` \Pp_2\Del_{\A^{\Del_{\C}3}}\Pp'_0` \Pp_2\Del_{\A^{\Del_{\C}3}}F_0]{1540}1a
\putmorphism(3160,500)(0,-1)[``F_2\Del_{\A^{\Del_{\C}3}}\Pp'_0]{500}1l
\putmorphism(-160,500)(0,-1)[``\Pp_1\Del_{\A^{\Del_{\C}3}}F_3]{500}1r
\putmorphism(-200,0)(1,0)[\Pp_1\Del_{\A^{\Del_{\C}3}}\Pp'_3` \Pp'_1\Del_{\A^{\Del_{\C}3}}\Pp'_3 ` F_1\Del_{\A^{\Del_{\C}3}}\Pp'_3]{1600}1b
\putmorphism(1400,0)(1,0)[\phantom{(X \ot (Y \ot U))}` \Pp'_2\Del_{\A^{\Del_{\C}3}}\Pp'_0` \alpha']{1600}1b
\efig}
\end{equation}
meaning that
\begin{equation} \eqlabel{morf diagram}
\alpha'(F_1\Del_{\A^{\Del_{\C}3}}F_3)\simeq(F_2\Del_{\A^{\Del_{\C}3}}F_0)\alpha
\end{equation}

\bigskip

On the other hand, given a morphism $(\Pp, \beta)\to(\Pp', \beta')$ in $\dul{Z}^2(\A,\dul{\Pic})$, it is given by an equivalence $F:\Pp\to\Pp'$ in $\A\Del_{\C}\A\x\Mod$
such that $\beta'\comp\delta_2(F)\simeq\beta$. Applying \leref{alfa dagger}, 2) we may write this as:
$$\beta=
\gbeg{4}{4}
\got{1}{\crta{1}} \got{1}{2} \got{1}{0} \got{1}{\crta{3}} \gnl
\gcl{1} \glmptb\gnot{\hspace{-0,4cm}\alpha^{-1}}\grmptb \gcl{1} \gnl
\gev \gev \gnl
\gob{1}{}
\gend\simeq
\gbeg{7}{4}
\got{1}{\crta{1}} \got{1}{2} \got{1}{0} \got{1}{\crta{3}} \gnl
\gbmp{F_1^{\dagger}} \gbmp{F_2} \gbmp{F_0} \gbmp{F_3^{\dagger}} \gnl
\gcl{1} \glmptb\gnot{\hspace{-0,4cm}\alpha'^{-1}}\grmptb \gcl{1} \gnl
\gev \gev \gnl
\gob{1}{}
\gend
$$
Tensor this from the left by $\Pp_1$ and the right by $\Pp_3$ and then apply $\crta{coev}_1$ and $coev_3$ from above, to get equivalently (by \leref{basic lema}, 1)
and since $coev$ is an equivalence):
$$
\gbeg{6}{4}
\got{1}{} \got{1}{} \got{1}{2} \got{1}{0} \gnl
\gdb \glmptb\gnot{\hspace{-0,4cm}\alpha^{-1}}\grmptb \gdb \gnl
\gcl{1} \gev \gev \gcl{1} \gnl
\gob{1}{1} \gvac{4} \gob{1}{3}
\gend\simeq
\gbeg{7}{6}
\got{1}{} \got{1}{} \got{1}{2} \got{1}{0} \gnl
\gdb \gcl{1} \gcl{1} \gdb \gnl
\gcl{1} \gbmp{F_1^{\dagger}} \gbmp{F_2} \gbmp{F_0} \gbmp{F_3^{\dagger}} \gcl{3} \gnl
\gcl{1} \gcl{1} \glmptb\gnot{\hspace{-0,4cm}\alpha'^{-1}}\grmptb \gcl{1} \gnl
\gcl{1} \gev \gev \gnl
\gob{1}{1} \gvac{4} \gob{1}{3}
\gend
\quad\Leftrightarrow\quad\quad
\gbeg{3}{5}
\got{1}{2} \got{1}{0} \gnl
\gcl{1} \gcl{1} \gnl
\glmptb\gnot{\hspace{-0,4cm}\alpha^{-1}}\grmptb \gnl
\gcl{1} \gcl{1} \gnl
\gob{1}{1} \gob{1}{3}
\gend\simeq
\gbeg{7}{7}
\gvac{1} \got{1}{2} \got{1}{0} \gnl
\gvac{1} \gcl{1} \gcl{1} \gnl
\gvac{1} \gbmp{F_2} \gbmp{F_0} \gnl
\gvac{1} \glmptb\gnot{\hspace{-0,4cm}\alpha'^{-1}}\grmptb \gnl
\glmp\gnot{\hspace{-0,4cm}F_1^{-1}}\grmptb \glmptb\gnot{\hspace{-0,4cm}F_3^{-1}}\grmp \gnl
\gvac{1} \gcl{1} \gcl{1} \gnl
\gvac{1} \gob{1}{1} \gob{1}{3}
\gend
$$
where we applied the dual basis axiom and the definition of the opposite functor. So we got 
$$(F_1\Del_{\A^{\Del_{\C}3}}F_3)\alpha^{-1}\simeq \alpha'^{-1}(F_2\Del_{\A^{\Del_{\C}3}}F_0)$$ 
which is equivalent to \equref{morf diagram}. We obtained that a morphism $(\Id, \Psi)$ in
$\dul{QMon}(\A)$ defines a morphism $F=\Psi_{\A,\A}:(\Pp, \beta)\to(\Pp', \beta')$ in $\dul{Z}^2(\A,\dul{\Pic})$. Observe that we also have the converse:
given such $F$, we showed that $\beta'\comp\delta_2(F)\simeq\beta$ is equivalent to \equref{alfa na P}. Then define $\Psi_{\M,\N}: \M\boxdot\N\to\M\boxdot'\N$
as $\Psi_{\M,\N}=F\Del_{\A\Del_{\C}\A}\Id: \Pp\Del_{\A\Del_{\C}\A}(\M\Del_{\C}\N) \to \Pp'\Del_{\A\Del_{\C}\A}(\M\Del_{\C}\N)$ (recall \equref{new tensor}).
Then the diagram \equref{Psi na prod} with this $\Psi$ commutes by \leref{EW morfisms}. This shows that there is a fully faithful functor
$\dul{QMon}(\A)\to\dul{Z}^2(\A,\dul{\Pic})$.

\medskip

This functor is also essentially surjective: given $(\Pp, \beta)\in\dul{Z}^2(\A,\dul{\Pic})$, define the tensor product by \equref{tensor functor} and
the associativity constraint $a$ by the identification with $\omega_{\M, \N, \Ll}\simeq\alpha\Del_{\A^{\Del_{\C} 3}}\Id_{\M\Del_{\C}\N\Del_{\C}\Ll}$ from \equref{assoc},
and $\alpha$ we define by
$$
\alpha=(\ev_1\Del_{\A^{\Del_{\C}3}}\Pp_2\Del_{\A^{\Del_{\C}3}}\Pp_0\Del_{\A^{\Del_{\C}3}}\crta\ev_3)(\Pp_1\Del_{\A^{\Del_{\C}3}}\beta^{-1}\Del_{\A^{\Del_{\C}3}}\Pp_3).
$$
\qed\end{proof}

We have seen that $\dul{Z}^2(\A,\dul{\Pic})$ is a symmetric monoidal category and we constructed the group of 2-cocycles taking the Grothendieck group
of the category. In order to do the same for the category $\dul{QMon}(\A)$, let us equip it with a monoidal stracture, so to have a monoidal equivalence of categories.
Pulling back the monoidal structure on $\dul{Z}^2(\A,\dul{\Pic})$ to $\dul{QMon}(\A)$ tells us to define the tensor product on it as follows. Given quasi-monoidal structures
determined by $(\boxdot, a)$ and $(\boxdot', a')$, define their product as $(\tilde\boxdot, \tilde a)$, such that $\tilde\boxdot$ is induced by
$\tilde\Pp=\Pp\Del_{\A\Del_{\C}\A}\Pp'\in\dul{\Pic}(\A\Del_{\C}\A)$, where $\boxdot$ is induced by $\Pp=\A\boxdot\A$ and similarly $\boxdot'$ by $\Pp'$.
Thus, given $\M, \N\in\A\x\Mod$ we define:
$$\M\tilde\boxdot\N={}_{\Delta}(\Pp\Del_{\A\Del_{\C}\A}\Pp')\Del_{\A\Del_{\C}\A}(\M\Del_{\C}\N).$$
Here $\tilde a$ (identified with its corresponding $\tilde\omega$) is induced by $\tilde\alpha=\alpha\Del_{\A\Del_{\C}\A}\alpha'$ (up to the twist of the factors),
where obviously $\alpha$ and $\alpha'$ are those coming from $a$ and $a'$.
We now clearly have:

\begin{thm}
There is an isomorphism of groups:
$$K_0(\dul{QMon}(\A))\iso Z^2(\A,\dul{\Pic}).$$
\end{thm}

Taking the quotient of $K_0(\dul{QMon}(\A))$ by the subgroup induced by the 2-coboundaries will lead us to the isomorphism with the second cohomology group.
This subgroup is given by quasi-monoidal structures $(\boxdot, a)$, where $\boxdot$ is induced by some $\Pp\in\dul{\Pic}(\A\Del_{\C}\A)$ and $a$ is induced by
some $\alpha$ which induces $\beta$ by \equref{beta}, and moreover $\Pp$ and $\beta$ are such that there is an equivalence $F: \Pp\to\delta_1(\N)$ in $\A\Del_{\C}\A\x\Mod$
for some $\N\in\dul{\Pic}(\A)$, so that $\beta\simeq\lambda_{\N}\comp\delta_2(F)$. Observe that $\delta_1(\N)=\N_0\Del_{\A\Del_{\C}\A}\N_1^{op}\Del_{\A\Del_{\C}\A}\N_2=
(\A\Del_{\C}\N)\Del_{\A\Del_{\C}\A}\N_1^{op}\Del_{\A\Del_{\C}\A}(\N\Del_{\C}\A)\simeq(\N\Del_{\C}\N)\Del_{\A\Del_{\C}\A}\N_1^{op}$.
If we denote by $QMon(\A)$ the quotioent of $K_0(\dul{QMon}(\A))$ by the described subgroup, we clearly have:

\begin{thm}
There is an isomorphism of groups:
$$QMon(\A)\iso H^2(\A,\dul{\Pic}).$$
\end{thm}

\bibliographystyle{amsalpha}

\begin{thebibliography}{AE}


\bibitem{AEG} N. Andruskiewitsch, P. Etingof, S. Gelaki, \emph{Triangular Hopf algebras with the Chevalley property}, Michigan Math. J. \textbf{49}, 277--298 (2001).


\bibitem{Bass} H. Bass, {\em Algebraic K-Theory}, Columbia University, W.A. Benjamin Inc. 1968.


\bibitem{Be} J. B\'enabou, \emph{Introduction to bicategories}, Reports of the Midwest Category Seminar, Lecture Notes in Mathematics \textbf{47}, 1-?7 Springer, Berlin 1967. 


\bibitem{Bo} F. Borceux,  \emph{Handbook of Categorical Algebra, Basic Category Theory (Encyclopedia of Mathematics and its Applications)},
Volume 1, Cambridge University Press 1994.


\bibitem{CF} S. Caenepeel, B. Femi\'c, {\em The Brauer Group of Azumaya Corings and the Second Cohomology Group},
K-theory {\bf 34}, 361-393 (2005).


\bibitem{DN}  A. Davydov, D. Nikshych, \emph{The Picard crossed module of a braided tensor category}, Algebra and Number Theory
{\bf 7}/6, 1365--1403 (2013).


\bibitem{DGNO1} V. Drinfeld, S. Gelaki, D. Nikshych, V. Ostrik, \emph{On braided fusion categories I}, Selecta Math. {\bf 16}/1, 1--119 (2010).


\bibitem{Eilenberg} S. Eilenberg, {\em Abstract description of some basic functors}, J. Indian Math. Soc. (N.S.) {\bf 24}, 1960, 231--234 (1961).


\bibitem{EG}  P. Etingof, S. Gelaki, \emph{The classification of finite-dimensional triangular Hopf algebras over an algebraically closed
field of characteristic 0}, Mosc. Math. J. {\bf 3}/1, 37--43 (2003).


\bibitem{EGNObook} P. Etingof, S. Gelaki, D. Nikshych and V. Ostrik, \emph{Tensor categories}, Mathematical Surveys and Monographs, Volume 205,
AMS, Providence, Rhode Island, 2015.

\bibitem{ENO} P. Etingof, D. Nikshych and V. Ostrik, \emph{Fusion categories and homotopy theory}, Quantum Topology \textbf{1}/3, 209--273 (2010).


\bibitem{EO} P. Etingof, V. Ostrik, \emph{Finite tensor categories}. Mosc. Math. J. \textbf{4}/3, 627--654	(2004).


\bibitem{FMM} B. Femi\'c, A. Mej\'ia, M. Mombelli, {\em Invertible bimodule categories over the representation category of a Hopf algebra.}
J. Pure Appl. Algebra, {\bf 218}/11, 2096--2118 (2014). 


\bibitem{th} {\sc B. Femi\'c}. {\em Azumaya corings, braided Hopf-Galois theory and Brauer groups}, Ph.D. thesis, available at: 
https://www.fing.edu.uy/~bfemic/. 


\bibitem{Femic2} B. Femi\'c, {\em Villamayor-Zelinsky sequence for symmetric finite tensor categories},  arXiv:1505.06504v2.


\bibitem{Femic3} B. Femi\'c, {\em Coring categories and Villamayor-Zelinsky sequence for symmetric finite tensor categories}, arXiv:1508.05420v2.


\bibitem{Ga1} C. Galindo, {\em Clifford theory for tensor categories}, J. Lond. Math. Soc. (2)/83, 57-78 (2011).


\bibitem{Ga} C. Galindo, \emph{Crossed product tensor categories}, J. Algebra \textbf{337}, 233--252 (2011).


\bibitem{Gr1} J. Greenough, \emph{Relative centers and tensor products of tensor and braided fusion categories}, J. Algebra \textbf{388}, 374--396 (2013).


\bibitem{Gr}  J. Greenough, \emph{Monoidal 2-structure of Bimodule Categories}, J. Algebra \textbf{324}, 1818--1859 (2010).


\bibitem{KV} M.M. Kapranov, V.A. Voevodsky, {\em 2-categories and Zamolodchikov tetrahedra equations}, 
Algebraic groups and their generalizations: quantum and infinite-dimensional methods (University Park, PA, 1991), 
Proc. Sympos. Pure Math. {\bf  56}/2, 177--259, Amer. Math. Soc., Providence, RI, 1994.



\bibitem{M} M. M\"uger, {\em On the Structure of Modular Categories}, Proc. London Math. Soc. {\bf 87}/3, 291--308 (2003).


\bibitem{Neu}  M. Neuchl, \emph{Representation theory of Hopf categories}, Ph.D. thesis, available at: math.ucr.edu/home/baez/neuchl.ps. 

\bibitem{Os1}  V. Ostrik, \emph{Finite tensor categories}, Mosc. Math. J. \textbf{4}/3, 627--654	(2004).

\bibitem{Watts} C. E. Watts, {\em Intrinsic characterizations of some additive functors}, Proc. Amer. Math. Soc. {\bf 11}, 5--8 (1960).


\end{thebibliography}

\end{document}